\documentclass[12pt]{article}
\usepackage{ae}
\usepackage[T1]{fontenc}
\usepackage[nottoc]{tocbibind}
\usepackage{amsmath}
\usepackage{amsthm}
\usepackage{amssymb}
\usepackage{braket} 
\usepackage{mathtools} 
\usepackage{dsfont} 
\usepackage{accents} 
\usepackage{upgreek}

\usepackage{color}

\numberwithin{equation}{section} 

\newcommand{\itref}[1]{\textup{(#1)}}

\makeatletter
\let\origthm@notefont\thm@notefont
\renewcommand{\thm@notefont}[1]{\origthm@notefont{#1\the\thm@headfont}}
\makeatother

\newtheorem{thm}{Theorem}[section]
\newtheorem{theorem}[thm]{Theorem}
\newtheorem{corollary}[thm]{Corollary}
\newtheorem{lemma}[thm]{Lemma}
\newtheorem{proposition}[thm]{Proposition}

\theoremstyle{definition}
\newtheorem{definition}[thm]{Definition}
\newtheorem{example}[thm]{Example}
\newtheorem{observation}[thm]{Observation}
\newtheorem{remark}[thm]{Remark}
\newtheorem{notation}[thm]{Notation}

\newcommand{\cat}[1]{{\boldsymbol{\mathsf{#1}}}} 
\newcommand{\Sets}{\cat{Set}}

\newcommand{\SMan}{\cat{SMan}}
\newcommand{\AoMan}{\cat{A_0 Man}}
\newcommand{\cAoMan}{\cat{\mathcal{A}_0 Man}}

\newcommand{\SAlg}{\cat{SAlg}}
\newcommand{\Gras}{\cat{\Lambda}}
\newcommand{\SWA}{\cat{SWA}}
\newcommand{\Grp}{\cat{Grp}}

\newcommand{\op}{^{\mathrm{op}}} 
\newcommand{\funct}[2]{\left[{#1},{#2}\right]} 
\newcommand{\cfunct}[2]{\left[{#1}\op,{#2}\right]} 
\newcommand{\dfunct}[2]{\left[\left[{#1},{#2}\right]\right]}
\newcommand{\subfunc}[2]{#1_{#2}}

\newcommand{\FOP}[1]{#1(\blank)} 
\newcommand{\fop}[2]{#1(#2)} 

\newcommand{\yon}{\mathcal{Y}} 
\newcommand{\ber}{\mathcal{B}} 
\newcommand{\sch}{\mathcal{S}} 

\newcommand{\nset}[1]{{\mathds{#1}}} 
\newcommand{\N}{\nset{N}}
\newcommand{\Z}{\nset{Z}}
\newcommand{\R}{\nset{R}}
\newcommand{\C}{\nset{C}}
\newcommand{\K}{\nset{K}}

\renewcommand{\to}{\mathchoice{\longrightarrow}{\rightarrow}{\rightarrow}{\rightarrow}} 
\let\mto\mapsto \renewcommand{\mapsto}{\mathchoice{\longmapsto}{\mto}{\mto}{\mto}} 

\newcommand{\id}{\mathds{1}} 
\newcommand{\pr}{\mathrm{pr}} 
\newcommand{\im}{\mathrm{Im}} 
\newcommand{\isom}{\cong} 

\newcommand{\Hom}{\mathrm{Hom}}
\newcommand{\HOM}{\underline{\Hom}}
\newcommand{\Der}{\mathrm{Der}}

\newcommand{\restr}[2]{{#1}_{|{#2}}} 

\newcommand{\sheaf}[1][O]{\mathcal{#1}} 
\newcommand{\stalk}[3][O]{\sheaf[#1]_{#2,#3}} 
\newcommand{\distr}[3][]{\sheaf_{#2,#3}^{#1*}} 
\newcommand{\germ}[1]{[#1]} 
\newcommand{\red}[1]{\widetilde #1} 
\newcommand{\rred}[1]{\innerrred #1} \newcommand{\innerrred}[1]{\widetilde{\widetilde{#1}}}

\newcommand{\topo}[1]{\abs{#1}} 
\newcommand{\maxid}{\mathcal{M}} 
\newcommand{\ev}{\mathrm{ev}} 

\newcommand{\Cinf}{\sheaf[C]^\infty} 
\newcommand{\Hol}{\sheaf[H]} 
\newcommand{\functions}{\sheaf[F]} 

\newcommand{\p}[1]{p(#1)} 

\newcommand{\nil}[1]{\accentset{\circ}#1} 

\newcommand{\ext}[2][]{\Lambda_{#1}(#2)} 
\newcommand{\extn}[1]{\Lambda_{#1}} 

\newcommand{\fseries}[2]{\mathfrak{A}_{#1|#2}} 
\newcommand{\kspoly}[2]{\K[#1|#2]} 

\newcommand{\funcpt}[1]{\underline #1} 

\newcommand{\X}{\mathrm{x}}
\newcommand{\Y}{\mathrm{y}}
\newcommand{\T}{\uptheta}

\newcommand{\nbd}{\protect\nobreakdash-\hspace{0pt}} 

\DeclarePairedDelimiter{\abs}{\lvert}{\rvert} 
\newcommand{\pair}[2]{\big\langle #1, #2 \big\rangle} 

\newcommand{\blank}{\mspace{2mu}{\cdot}\mspace{2mu}} 
\newcommand{\pd}[2]{\frac{\partial #1}{\partial #2}} 
\newcommand{\supp}{\mathrm{supp}} 

\renewcommand{\phi}{\varphi}
\renewcommand{\theta}{\vartheta}
\renewcommand{\epsilon}{\varepsilon}

\newcommand{\cA}{\mathcal{A}}
\newcommand{\cO}{\mathcal{O}}
\newcommand{\cF}{\mathcal{F}}
\newcommand{\cG}{\mathcal{G}}
\newcommand{\cU}{\mathcal{U}}
\newcommand{\F}{\boldsymbol{F}}
\newcommand{\f}{\boldsymbol{f}}
\newcommand{\bU}{\accentset{\frown}{U}}
\newcommand{\bV}{\accentset{\frown}{V}}

\newcommand{\al}{\alpha}

\newcommand{\be}{\beta}

\newcommand{\lra}{\longrightarrow}

\begin{document}

\centerline{\Large\bf A Comparison of the functors }

\medskip

\centerline{\Large\bf of points of Supermanifolds}

\bigskip

\centerline{L. Balduzzi$^\natural$, C. Carmeli$^\sharp$, R. Fioresi$^\flat$}

\bigskip

\centerline{\it $^\natural$Dipartimento di Fisica, Universit\`a di Genova, and INFN, sezione di Genova}
\centerline{\it Via Dodecaneso 33, 16146 Genova, Italy}
\centerline{\footnotesize e-mail: luigi.balduzzi@infn.ge.it}

\bigskip

\centerline{\it $^\sharp$DIME, Universit\`a di Genova, and INFN, sezione di Genova}
\centerline{\it Via Cadorna 2,  17100 Savona,   Italy}
\centerline{\footnotesize e-mail: carmeli@diptem.unige.it}

\bigskip

\centerline{\it $^\flat$Dipartimento di Matematica, Universit\`a di Bologna}
\centerline{\it Piazza di Porta San Donato 5, 40127 Bologna, Italy}
\centerline{\footnotesize e-mail: fioresi@dm.unibo.it}

\begin{abstract}
We study the functor of points and different local functors of points 
for smooth and
holomorphic supermanifolds, providing characterization theorems and
fully discussing the representability issues. In the end we examine
applications to differential calculus including the transitivity
theorems.
\end{abstract}

\tableofcontents

\section{Introduction}

This paper is devoted to understand the approach to supergeometry
via different \emph{local functors of points} for
both the differential and the holomorphic category.

\medskip

In most of the treatments of supergeometry 
(see among many others \cite{BBHR, BL, Kostant, Berezin, Leites, Manin, DM,
Varadarajan}) supermanifolds are understood as classical
manifolds 
with extra anticommuting coordinates, introduced as odd elements in the
structure sheaf.

\medskip

Only later in \cite{Manin, DM}, the functorial
language starts to be used systematically and the functor of
points approach becomes a powerfull device allowing, among
other things to give a rigorous meaning to otherwise just formal
expressions and to recover some geometric intuition
otherwise lost. In
this approach, a supermanifold $M$ is fully recovered by
the knowledge of its functor
of points, $S \mapsto \fop{M}{S} \coloneqq
\Hom(S,M)$, which assigns to each supermanifold $S$,
the set of the $S$\nbd points of $M$, $\fop{M}{S}$. This is
in essence the content of Yoneda's Lemma
(see, for example, \cite{MacLane}).

\medskip

Whereas in the sheaf theoretic approach the nilpotent coordinates are
introduced by enlarging the sheaf of a classical
manifold without modifying the underling topological space,
another possible approach to supermanifolds theory consists in
introducing new local models for the underlying set itself
(\cite{Batchelor, Rogers80, DeWitt, Shvarts}).
In this setting, supermanifolds are
obtained by gluing domains of the form $\Lambda_0^p\times
\Lambda_1^q$, where $\Lambda_0$ and $\Lambda_1$ are the even and odd
part of some Grassmann algebra $\Lambda$. This idea is
actually the original physicists approach to supergeometry and
only later its mathematical foundations were
developed from different perspectives. In particular, in
\cite{Batchelor} M. Batchelor, following B. DeWitt, considers
supermanifolds which are locally isomorphic to $(\extn{L})_0^p
\times (\extn{L})_1^q$ where $\extn{L}$ is an arbitrary fixed
Grassmann algebra with $L>q$ generators; the topology is non
Hausdorff and the smooth functions over a supermanifold are
defined in an ad hoc way in order to obtain the same
class of morphisms as in Kostant's original approach \cite{Kostant}. 

From another point of
view, A. S. Shvarts and A. A. Voronov (see \cite{Shvarts, Voronov})
consider simultaneously every finite dimensional Grassmann algebra
$\Lambda$ and local models that vary functorially with $\Lambda$;
the topology is the classical Hausdorff one and  morphisms between
superdomains are defined as appropriate natural transformations
between them. Both approaches turn out to
be equivalent to Kostant's original formulation, as it is proved in
\cite{Batchelor} and \cite{Voronov}. 

Other possible frameworks
involving a wider class of supermanifold morphisms have been
considered in the literature (see, for example, \cite{DeWitt} and
\cite{Rogers80}). In particular in DeWitt's approach, supermanifolds
have a non Hausdorff topology and  are modeled on a
Grassmann algebra with countable infinitely many generators. For a
detailed review of some of these approaches we refer the reader to
\cite{BBHR, Rogers07}.

\medskip

This paper represents the generalization
to the holomorphic setting of the paper \cite{bcf}; moreover
it contains more observations, examples and proofs of the
results, for both the differential
and the holomorphic setting, that overall make the whole material more
accessible. 
We have also made an effort to compare our treatment with
Batchelor's one and also with Shvarts and Voronov point of view,
which is directly inspiring our definition of 
Weil-Berezin functor of points. As far as we know, there is
no organic treatment of such local functors of points even in 
the differential setting, though we are aware that it is
somehow common knowledge. It is our hope that our work will
fill such a gap, providing in addition a treatment which is
able to accomodate also the category of holomorphic supermanifolds,
which is substantially different from the differential one.
In fact, as far as the holomorphic category is concerned,
we believe that our approach is novel and shows that the
local functor of points can be employed easily beyond 
what it was originally developed for. 
Furthermore we prove a representability theorem which enables to
single out among a certain class of functors, those which
are the local functor of points of supermanifolds. To our knowledge,
this appears for the first time in this work.
Despite the abstraction, the importance of such a theorem should
not be overlooked. Very often physicists resort to the functor
of points to define supergeometric objects, hence the representability
issue becomes essential to show that there is a supergeometric
object corresponding to the functor in exam.
 
\medskip

We have made an effort to make our work self-contained as much
as possible, though we are going to rely for the main results
of supergeometry, like the Chart's theorem or Hadamard's lemma to the
many available references on the subject, providing punctual references
whenever we need them.

\medskip

Our paper is organized as follows.

\medskip

In section \ref{sec:basic_def} we review some  basic definitions of
supergeometry like the definition of superspace,
supermanifold and its associated functor of points as it is
discussed in J. Bernstein's notes by P. Deligne and J. W. Morgan \cite{DM}.
We briefly recall the representability problem and we state
the representability theorem for supermanifolds.

\medskip

In section \ref{sec:SWA_A-points} we introduce super Weil algebras
with their basic properties.
The basic observation is that super Weil
algebras are, by definition, local algebras. At a heuristic level
it is clear they are well suited to study the local
properties of a supermanifold, or, in other words, the properties of
the stalks at the various points of a supermanifold.
Once we define the functor
$A \mapsto M_A$ from the category of super Weil algebras to the
category of sets, it  is only natural to look for an
analogue of Yoneda's lemma, in other words for a result that
allows to retrieve the supermanifold from its local functor of points.
It turns out (see subsection
\ref{subsec:A-point_nat_trans}) that, as it is stated,
this result is not true; in order for the
local functor of points to be able to characterize
the supermanifold, its category of arrival (sets, as we defined it)
needs to be suitably specialized by
giving to each set $M_A$ an extra structure.

\medskip

In section \ref{sec:WB_functor-Sh_emb}, we discuss the modifications
we need to introduce in order to obtain a
bijective correspondence between supermanifold morphisms and natural
transformations between the local functors of points.
Following closely what is proved in \cite{Shvarts,
Voronov}, it turns out that it is necessary to endow the set $M_A$
with the structure of an $A_0$\nbd smooth manifold (see definition
\ref{def:Az_man}). We call the functor $A\to M_A$, with
$M_A$ an $A_0$\nbd smooth manifold, the
\emph{Weil--Berezin functor of $M$}. The main result is that,
in such a context, the analogue of
Yoneda's lemma holds (see theorem \ref{theor:full_and_faithful}),
and as a consequence supermanifolds embed in a full and faithful way into
the category of Weil--Berezin functors
(\emph{Shvarts embedding}).

In analogy with the classical theory, it is only natural to ask
under which conditions a generic local functor is representable,
meaning it is the Weil--Berezin 
functor of points of a supermanifold (strictly speaking we are
abusing the word ``representable'').
We prove a representability theorem for generic
functors from the category of super Weil algebras to the category of
$A_0$\nbd smooth manifolds.

We end the section by giving  an account of  the functor of
$\Lambda$\nbd points originally described by Shvarts, which is a
restriction of the Weil--Berezin functor to Grassmann algebras which form
a full subcategory of super Weil algebras. We also describe
Batchelor's approach, providing a comparison between these two approaches
and the Weil-Berezin functor discussed previously.

\medskip

In section \ref{sec:diff_calc} we examine some aspects of super
differential calculus on supermanifolds in the language of the
Weil--Berezin functor. We describe
the finite support distributions over the supermanifold
$M$ and their relations with the $A$\nbd points of $M$,
establishing a connection between our treatment and
Kostant's seminal approach to supergeometry.

We also prove the super version of the Weil transitivity
theorem, which is a key tool for the study of the infinitesimal
aspects of supermanifolds, and we apply it to define
the ``tangent functor'' of $A\mapsto M_A$.

\paragraph{Acknowledgements.} We want to thank 
prof.\ G. Cassinelli, prof.\ A. Cattaneo, prof.\ M.
Duflo, prof. P. Michor, and prof.\ V. S. Varadarajan for helpful
discussions. We also wish to thank the Referee,
whose comments have helped us to improve our manuscript.

\section{Basic Definitions} \label{sec:basic_def}

In this section we recall few basic definitions in supergeometry.
Our main references are \cite{Kostant, Manin, DM, Varadarajan}.

\subsection{Supermanifolds} \label{subsec:SMan}

Let the ground field $\K$ be $\R$ or $\C$.

\medskip

A \emph{super vector space} is a $\Z_2$\nbd graded vector space,
i.~e.\ $V=V_0 \oplus V_1$; the elements in $V_0$ are called
\emph{even}, those in $V_1$ \emph{odd}. An element $v \neq 0$
in $V_0 \cup V_1$ is said \emph{homogeneous} and $\p{v}$ denotes its
parity: $\p{v}= 0$ if $v \in V_0$, $\p{v}=1$ if $v \in V_1$.
$\K^{p|q}$ denotes the super vector space $\K^p \oplus \K^q$.

\medskip

A \emph{superalgebra} $A$ is an algebra that is also a super vector
space, $A=A_0 \oplus A_1$, and such that $A_i A_j \subseteq
A_{i+j\pmod 2}$. $A_0$ is an algebra, while $A_1$ is an $A_0$\nbd
module.
$A$ is said to be \emph{commutative} if for any two homogeneous
elements $x$ and $y$
\[
    xy = (-1)^{\p{x}\p{y}} yx \text{.}
\]
The category of commutative superalgebras is denoted by
$\SAlg_\K$ or simply by $\SAlg$ when no confusion is possible.
From now on all superalgebras are assumed to be
commutative unless otherwise specified.

\begin{definition}
A \emph{superspace} $S=(\topo{S}, \sheaf_S)$ is a topological space
$\topo{S}$, endowed with a sheaf of superalgebras $\sheaf_S$ such
that the stalk at each point  $x \in \topo{S}$, denoted by
$\sheaf_{S,x}$, is a local superalgebra.
\end{definition}

\begin{definition}
A \emph{morphism} $\phi \colon S \to T$ of superspaces is a pair
$(\topo{\phi}, \phi^*)$, where $\topo{\phi} \colon \topo{S} \to
\topo{T}$ is a continuous map of topological spaces and $\phi^* \colon \sheaf_T
\to \topo{\phi}_* \sheaf_S$, called \emph{pullback}, is such that
$\phi_x^*(\maxid_{\topo{\phi}(x)}) \subseteq \maxid_x$ where
$\maxid_{\topo{\phi}(x)}$ and $\maxid_{x}$ denote the maximal ideals in
the stalks $\sheaf_{T,\topo{\phi}(x)}$ and $\sheaf_{S,x}$
respectively.
\end{definition}

\begin{example}[The smooth local model] \label{superspace-rs}
The superspace $\R^{p|q}$ is the topological space $\R^p$ endowed
with the following sheaf of superalgebras. For any open set $U
\subseteq \R^p$ define
\[
    \sheaf_{\R^{p|q}}(U) \coloneqq \Cinf_{\R^p}(U) \otimes \ext[\R]{\theta_1,\dots,\theta_q}
\]
where $\ext[\R]{\theta_1,\dots,\theta_q}$ is the real
exterior algebra
(or \emph{Grassmann algebra}) generated by the $q$ variables
$\theta_1,\dots,\theta_q$ and $\Cinf_{\R^p}$ denotes the $\Cinf$
sheaf on $\R^p$.
\end{example}

\begin{example}[The holomorphic local model]
The superspace $\C^{p|q}$ is the topological space $\C^p$ endowed
with the following sheaf of superalgebras. For any open set $U
\subseteq \C^p$ define
\[
    \sheaf_{\C^{p|q}}(U) \coloneqq \Hol_{\C^p}(U) \otimes \ext[\C]{\theta_1,\dots,\theta_q}
\]
where $\Hol_{\C^p}$ denotes the holomorphic sheaf on $\C^p$
and  $\ext[\C]{\theta_1,\dots,\theta_q}$ is now the complex
exterior algebra.
\end{example}

By a common abuse of notation, $\K^{m|n}$ denotes both the
super vector space $\K^m \oplus \K^n$ and the superspace
defined above.

\begin{definition} \label{def:supermanifold}
A smooth (resp.\ holomorphic) \emph{supermanifold} of dimension
$p|q$ is a superspace $M=(\topo{M},\sheaf_M)$ which is locally
isomorphic to $\R^{p|q}$ (resp.\ $\C^{p|q}$), i.~e.\ for all $x \in
\topo{M}$ there exist  open sets $x \in V_x \subseteq \topo{M}$ and
$U \subseteq \R^{p}$ (resp.\ $\C^{p}$) such that:
\[
    \restr{\sheaf_{M}}{V_x} \isom \restr{\sheaf_{\R^{p|q}}}{U}
    \qquad \text{(resp.\ $\restr{\sheaf_{M}}{V_x} \isom \restr{\sheaf_{\C^{p|q}}}{U}$).}
\]
The sheaf $\sheaf_M$ is called the \emph{structure sheaf} of the
supermanifold $M$.
A \emph{morphism} of supermanifolds is simply a morphism of
superspaces. $\SMan_\K$ (or simply $\SMan$) denotes the category of supermanifolds.
In particular supermanifolds of the form
$(U,\restr{\sheaf_{\K^{p|q}}}{U})$ are called \emph{superdomains}. 
{If $\topo{U}\subseteq  \C^n$ is a domain of holomorphy  (see, for example, \cite{Horm}), then we say that
$(\topo{U},\restr{\sheaf_{\C^{p|q}}}{U})$ is a \emph{ Stein superdomain}.  }

If $U$ is open in $\topo{M}$, $(U,\restr{\sheaf_M}{U})$ is also a
supermanifold and it is called the \emph{open supermanifold
associated to $U$}. We shall often refer to it just by $U$, whenever
no confusion is possible.
\end{definition}

{
\begin{remark}
\cite[Corollary 2.5.6]{Horm} gives many examples of domains of holomorphy in $\C^m$ and hence many examples of Stein superdomains. In particular $\C^{m|n}$ and the open supermanifolds associated with polydiscs are Stein superdomains. We recall that, given $(z_1,\ldots,z_m)\in \C^m$ and $m$ positive real numbers $(r_1,\ldots,r_m)$ the polydisc with center $(z_1,\ldots,z_m)$ and polyradius $(r_1,\ldots,r_m)$ is
\[
P(\overline{z}_1,\ldots,\overline{z}_n;r_1,\ldots,r_n)\coloneqq \set{(z_1,\ldots,z_n)\in \C^n\,|\, \abs{z_i-\overline{z}_i}<r_i \quad \forall 1\leq i\leq n}
\]

Since the Stein property 
is stable under biholomorphic morphisms, it is easy to see that every complex supermanifold admits an atlas of  Stein superdomains.
\end{remark}
}

\medskip

In order to avoid duplications and heavy notations, we will simply
refer to supermanifolds when the distinction between the smooth and
the holomorphic case is immaterial. Moreover if $M$ is a
supermanifold, we will denote by $\sheaf(M)$ the superalgebra
$\sheaf_M(\topo{M})$ of global sections on $M$.

\medskip

Suppose $M$ is a supermanifold and
$U$ is an open subset of $\topo{M}$. Let $\sheaf[J]_M(U)$ be the
ideal of the nilpotent elements of $\sheaf_M(U)$.
$\sheaf_M/\sheaf[J]_M$ defines a sheaf of purely even
algebras over $\topo{M}$ locally isomorphic to $\Cinf(\R^p)$ (resp.\
$\Hol(\C^p)$). Therefore $\red{M}
\coloneqq (\topo{M},\sheaf_M/\sheaf[J]_M)$ defines a classical
manifold, called the \emph{reduced manifold} associated to $M$. The
projection $s \mapsto \red{s} \coloneqq s + \sheaf[J]_M(U)$, with $s
\in \sheaf_M(U)$, is the pullback of the embedding $\red{M} \to M$.

\medskip

In the following we denote by $\ev_x(s) \coloneqq \red{s}(x)$ the
\emph{evaluation} of $s$ at $x \in U$. It is also possible to check
that $\topo{\phi}^*(\red{s}) = \red{{\phi^*(s)}}$, so that the
morphism $\topo{\phi}$ is automatically smooth (resp.\ holomorphic).
Moreover since the maximal ideal $\maxid_x$ in the stalk
$\stalk{M}{x}$ is given by the germs of sections whose value at $x$
is zero, we have that the locality condition in the case of
supermanifold  morphisms is automatically satisfied.

\medskip

There are several equivalent ways to assign a morphism between two
supermanifolds. The following result can be found in
\cite[ch.~4]{Manin}.

{
\begin{proposition}[Chart theorem] \label{prop:morphisms}
Let $U$ and $V$ two smooth or holomorphic superdomains, i.~e.\ two
open subsupermanifolds of $\K^{p|q}$ and $\K^{m|n}$ respectively.
There is a bijective correspondence between
\begin{enumerate}
    \item \label{item::1}
    superspace morphisms $U \to V$;
    \item \label{item::2}
    the set of pullbacks of a fixed coordinate system on
    $V$, i.~e.\ $(m|n)$\nbd uples
    \[
        (s_1,\ldots,s_m,t_1,\ldots,t_n) \in
        \sheaf(U)_0^m \times \sheaf(U)_1^n
    \]
    such that $\big(\red{s_1}(x),\ldots,\red{s_m}(x)\big)
    \in \topo{V}$ for each $x \in \topo{U}$.
\end{enumerate}
\end{proposition}
}

Any supermanifold morphism $M \to N$ is then uniquely determined by a
collection of local maps, once  atlases on
$M$ and $N$ have been fixed.  A morphism can hence be given by
describing it in local coordinates.

In the smooth category a further simplification
occurs: we can assign a morphism between supermanifolds
by assigning the pullbacks of the global sections
(see \cite[\S~2.15]{Kostant}), i.~e.\
\begin{equation} \label{eq:pullback_of_global_sections}
    \Hom_{\SMan_\R}(M,N) \isom \Hom_{\SAlg_\R} \big( \sheaf(N) , \sheaf(M) \big) \text{.}
\end{equation}
The  essential point here is that, borrowing some terminology from
algebraic geometry, smooth supermanifolds are an ``affine''
category. By this we mean that the knowledge of the superalgebra of
global sections allows us to fully reconstruct the supermanifold
obtaining its structure sheaf by a localization
procedure (see, for example, \cite{BBHR} and \cite{CF}).

\medskip

To end our very brief summary of supergeometric definitions 
and results in the superspace context, we shall
recall the superversion of Hadamard's lemma, which sheds light on
the structure of the stalk of the structure sheaf of a supermanifold.
Together with the Chart's theorem, these
are key results and we shall make frequent use of them in our work.

\begin{lemma}[Hadamard's lemma] \label{lemma:polynomials}
Suppose $M$ is a (smooth or holomorphic) supermanifold, $x \in
\topo{M}$ and $\set{x_i,\theta_j}$ is a system of coordinates in a
neighborhood of $x$. Denote as usual by $\maxid_x$ the ideal of
the germs of sections whose value at $x$ is zero. For each $\germ{s}
\in \stalk{M}{x}$ and $k \in \N$ there exists a polynomial $P$ in
$\germ{x_i}$ and $\germ{\theta_j}$ such that
\[
    \germ{s} - P \in \maxid_x^k \text{.}
\]
\end{lemma}

\begin{proof}
The holomorphic case is trivial since the stalk at $x$ identifies
with convergent power series. The proof in the smooth case can be
found, for example, in \cite[\S~2.1.8]{Leites} or
\cite[ch.~4]{Varadarajan}.
\end{proof}

The theory of supermanifolds resembles very closely the classical
theory. One can, for example, define tangent bundles, vector fields
and the differential of a morphism similarly to the classical case.
For more details see \cite{Kostant, Leites, Manin, DM, Varadarajan}.

\subsection{The functor of points} \label{subsec:functorofpoints}

Due to the presence of nilpotent elements in the structure sheaf of
a supermanifold, supergeometry can also be equivalently
and very effectively studied using
the language of \emph{functor of points}, a very useful tool in
algebraic geometry applications.
We briefly review it; the interested reader can consult \cite[\S~2.9]{DM} or
\cite[\S~4.5]{Varadarajan}.

\medskip

Let us first fix some notation we will use
throughout the paper. If $\cat{A}$ and $\cat{B}$ are two categories,
$\funct{\cat{A}}{\cat{B}}$  denotes the category of functors between
$\cat{A}$ and $\cat{B}$. Clearly, the morphisms in
$\funct{\cat{A}}{\cat{B}}$ are the natural transformations. Moreover
we denote by $\cat{A}\op$ the \emph{opposite category} of
$\cat{A}$, so that the category of contravariant functors between
$\cat{A}$ and $\cat{B}$ is identified with
$\cfunct{\cat{A}}{\cat{B}}$. For more details we refer to
\cite{MacLane}.

\begin{definition}
Given a supermanifold $M$, we define its \emph{functor of points}
\[
    \FOP{M} \colon \SMan\op \to \Sets
\]
as the functor from the opposite category of supermanifolds to the
category of sets defined on the objects as
\[
    \fop{M}{S} \coloneqq \Hom(S,M)
\]
and on the morphisms according to
\[
    \begin{aligned}
        \fop{M}{\phi} \colon \fop{M}{S} &\to \fop{M}{T} \\
        f &\mapsto f \circ \phi
    \end{aligned}
\]
where $\phi \colon T \to S$.

The elements in $\fop{M}{S}$ are also called the \emph{$S$\nbd points}
of $M$.
\end{definition}

Given two supermanifolds $M$ and $N$, Yoneda's lemma establishes
a bijective correspondence
\begin{eqnarray*}
    \Hom_{\SMan}(M,N) & \longleftrightarrow &
    \Hom_{\cfunct{\SMan}{\Sets}} \big( \FOP{M}, \FOP{N} \big)
\end{eqnarray*}
between the morphisms $M \to N$ and the natural transformations
$\FOP{M} \to \FOP{N}$ (see \cite[ch.~3]{MacLane} or
\cite[ch.~6]{EH}). This allows us to view a morphism of supermanifolds
as a family of morphisms $\fop{M}{S} \to \fop{N}{S}$ depending
functorially on the supermanifold $S$. In other words, Yoneda's
lemma provides an immersion
\[
    \yon \colon \SMan \to \cfunct{\SMan}{\Sets}
\]
of $\SMan$ into $\cfunct{\SMan}{\Sets}$ that is full and faithful.
There are however objects in
$\cfunct{\SMan}{\Sets}$ that do not arise as the functors of points of a
supermanifold. We say that a functor $\cF \in \cfunct{\SMan}{\Sets}$ is
\emph{representable} if it is isomorphic to the functor of points of
a supermanifold.

\begin{observation}
We first notice that for any supermanifold $M$ the set
$\fop{M}{\K^{0|0}} = \Hom_\SMan(\K^{0|0},M) \isom \topo{M}$ as
sets, since $\K^{0|0}$ is just a point. So the functor of points
allows us to recover the set of the points of the topological
space $\topo{M}$ underlying $M$.
The knowledge of the set-theoretical points $\fop{M}{\K^{0|0}}$ however
is far from enough to reconstruct the supermanifold $M$ and
this is because of
two distinct reasons:
\begin{enumerate}
    \item All the elements of $\fop{M}{\K^{0|0}}$ annihilate
    the nilpotent part of the sheaf, so they give us no information on the odd
    part of the structure sheaf.
    \item The functor $\FOP{M}$ takes values in the category of sets,
    hence $\fop{M}{\K^{0|0}}$ is just a set
    and does not contain any information on the differentiable
    structure of $M$, even in the classical setting.
\end{enumerate}

The supermanifold $M$ is then recaptured only from the knowledge of
its $S$\nbd points, for \emph{all} the supermanifolds $S$.
\end{observation}

We now want to recall a representability criterion, which allows to
single out, among all the functors from the category of
supermanifolds to sets, those that are representable, i.~e.\ those
that are isomorphic to the  functor of points of a supermanifold. In
order to do this, we need to generalize the notion of \emph{open
submanifold} and of \emph{open cover}, to fit this more general
functorial setting.

\begin{definition}
Let $\cU$ and $\cF$ be two functors $\SMan\op \to \Sets$. $\cU$ is a
\emph{subfunctor} of $\cF$ if $\cU(S) \subseteq \cF(S)$ for all $S
\in \SMan$ and this inclusion is a natural transformation. We denote
it by $\cU \subseteq \cF$.

We say that $\cU$ is an \emph{open subfunctor} of $\cF$ if for all
supermanifolds $T$ and all natural transformations $\alpha \colon
\FOP{T} \to \cF$, $\alpha^{-1}(\cU)=\FOP{V}$, where $V$ is open in
$T$. If $\cU$ is also representable we say that $\cU$ is an
\emph{open supermanifold subfunctor}.

Let $\cU_{i}$ be open supermanifold subfunctors of $\cF$. We say that
$\set{\cU_{i}}$ is an \emph{open cover} of a functor $\cF \colon \SMan\op
\to \Sets$ if
for all supermanifolds $T$ and all natural
transformations $\alpha \colon \FOP{T} \to \cF$,
$\alpha^{-1}(\cU_i)=\FOP{V_i}$ and $V_i$ cover $T$.
\end{definition}

Any functor $\cF \colon \SMan\op \to \Sets$ when restricted to the
category of open subsupermanifolds of a given supermanifold $T$
defines a presheaf over $\topo{T}$.

\begin{definition}
A functor  $\cF \colon \SMan\op \to \Sets$ is said to be \emph{a
sheaf} if it has the sheaf property, that is,
if $\set{T_i}$ is an open cover of a supermanifold $T$ and we
have a family $\set{\alpha_i}$, $\alpha_i \in \cF(T_i)$, such that
$\restr{\alpha_i}{T_i \cap T_j} = \restr{\alpha_j}{T_i \cap T_j}$, then
there exists a unique $\alpha \in \cF(T)$ mapping to each $\alpha_i$
(for more details see \cite{Vistoli}).

\medskip

In particular, when $\cF$ is restricted to the
open subsupermanifolds of a given supermanifold $T$, it is a sheaf
over $\topo{T}$.
\end{definition}

We are ready to state a representability criterion which gives
necessary and sufficient conditions for a functor from $\SMan\op$ to
$\Sets$ to be representable. This is a very formal result and for
this reason it holds as it is for very different categories as
smooth and holomorphic supermanifolds and even for superschemes (for
more details on this category see \cite{CF}). A complete description
of the classical representability criterion in algebraic geometry
can be found in \cite[ch.~1]{DG}.
For the super setting see \cite{FLV}.

\begin{theorem}[Representability criterion] \label{theor:representability}
A functor $\cF \colon \SMan\op \to \Sets$ is representable if and
only if:
\begin{enumerate}
    \item $\cF$ is a sheaf, i.~e.\ it has the sheaf property;
    \item $\cF$ is covered by open supermanifold subfunctors $\set{\cU_i}$.
\end{enumerate}
\end{theorem}

\section{Super Weil algebras and $A$\nbd points} \label{sec:SWA_A-points}

In this section we introduce the category $\SWA$ of super Weil
algebras. These are finite dimensional commutative superalgebras
with a nilpotent graded ideal of codimension one.  Super
Weil algebras are the basic ingredient in the definition of the
Weil--Berezin functor and the Shvarts embedding. The easiest
examples of super Weil algebras are Grassmann algebras. These are the
only super Weil algebras that can be interpreted as algebras of
global sections of supermanifolds, namely $\K^{0|q}$.
Given a supermanifold $M$, we want to define
a functor $M_{(\blank)} \colon \SWA
\to \Sets$ assigning to each super Weil algebra $A$ the set of
$A$\nbd points $M_{A}$.
If $A$ is a Grassmann algebra, the
$A$\nbd points of $M$ are identified with the usual $\K^{0|q}$\nbd
points in the functor of points
language described in the previous section.
Unfortunately this functor is not adequate to fully describe the
supermanifold $M$; as we shall see at the end of this section,
the arriving category needs to have an additional structure in
order for $M_{(\blank)}$ to contain the same information as $M$.
This is due, as we shall see, to the {\sl local} nature of $M_{(\blank)}$ .

\subsection{Super Weil algebras} \label{subsec:SWA}

We now define the category of \emph{super Weil algebras}. The
treatment follows closely that contained in \cite[\S~35]{KMS} for
the classical case.

\begin{definition}
We say that $A$ is a (real or complex) \emph{super Weil algebra} if
it is a commutative unital superalgebra over $\K$ and
\begin{enumerate}
    \item $\dim A < \infty$,
    \item $A=\K \oplus \nil{A}$,
    \item $\nil{A}=\nil{A}_0 \oplus \nil{A}_1$ is a graded nilpotent ideal.
\end{enumerate}
The category of super Weil algebras is denoted by $\SWA$.

We also define the \emph{height} of $A$ as the lowest $r$ such that
$\nil{A}^{r+1}=0$ and the \emph{width} of $A$ as the dimension of
$\nil{A}/\nil{A}^2$.
\end{definition}

Notice  that
super Weil algebras are local superalgebras, i.~e.\ they contain
a unique maximal graded ideal.

\begin{remark}
As a direct consequence of the definition, each super Weil algebra
has an associated short exact sequence:
\[
    0 \to \K \stackrel{j_A}{\to} A = \K \oplus \nil{A} \stackrel{\pr_A}{\to} A/\nil{A} \isom \K \to 0 \text{.}
\]
Clearly the sequence splits and each $a \in A$
can be written uniquely as
\[
    a = \red{a} + \nil{a}
\]
with $\red{a} \in \K$ and $\nil{a} \in \nil{A}$.
\end{remark}

\begin{example}[Dual numbers] \label{example:DN}
The simplest example of super Weil algebra in the classical setting is
$\K(x)=\K[x]/\langle x^2 \rangle$ the algebra of dual numbers. Here
$x$ is an even indeterminate which is nilpotent of degree two.
\end{example}

\begin{example}[Super dual numbers] \label{example:SDN}
The simplest non trivial example of super Weil
algebra in the super setting is
$\K(x,\theta) = \K[x,\theta] / \langle x^2,x\theta,\theta^2 \rangle$
where $x$ and $\theta$ are respectively even and odd indeterminates.
\end{example}

\begin{example}[Grassmann algebras]
The polynomial algebra in $q$ odd variables $\ext{\theta_1,\dots,
\theta_q}$ is another example of super Weil algebra.
Grassmann algebras are actually a full subcategory of
$\SWA$.
\end{example}

Let
\[
    \kspoly{k}{l} \coloneqq \K[t_1,\ldots,t_k] \otimes \ext{\theta_1,\ldots,\theta_l}
\]
denote the superalgebra of polynomials on $\K$ in $k$ even and $l$
odd variables.
$\kspoly{k}{l}$ is not a super Weil algebra, unless $k=0$, however
every finite dimensional graded quotient
$\kspoly{k}{l} / J$, with $J$ graded ideal, is a super Weil algebra.

\medskip

The next lemma gives a characterization of super Weil algebras,
which is going to be very important for our treatment.

\begin{lemma} \label{lemma:SWA}

The following are equivalent: \\
1. $A$ is a super Weil algebra;\\
2. $A \cong \kspoly{k}{l} /I$ for a suitable graded ideal $I
\supseteq
\langle t_1,\ldots,t_k, \theta_1, \ldots, \theta_l \rangle^k$.\\
3. $A \cong \sheaf_{\K^{p|q},0}/J$ for suitable $p,q$ and
$J$ graded ideal
containing a power of the maximal ideal $\maxid_0$ in the stalk
$\sheaf_{\K^{p|q},0}$;\\
\end{lemma}

\begin{proof} $(1) \implies (2)$.
Assume first $A$ is a super Weil algebra, generated
by the homogeneous elements $a_1 \dots a_n$. 
Then by universality we have a morphism:
$\K[X_1 \dots X_n] \lra A$, $X_i \mapsto a_i$, hence
$A \cong  \K[X_1 \dots X_n]/I$. Now we show
$I \supset (X_1 \dots X_n)^k$.
Let $n_i$ be such that $a_i^{n_i}=0$. So we have
$X_i^{n_i} \in I$. Let $N=max\{n_i\}$. We claim that
$(X_1 \dots X_n)^{nN} \subset I$. In fact
$(X_1 \dots X_n)^{nN}$ is generated by
$X_1^{j_1} \dots X_n^{j_n}$, with $j_1+\dots+j_n=nN$.
Since $j_1 \dots j_n$ are $n$ non negative integers whose
sum is $nN$, we have that at least one of them $j_l$ is greater than
$N$, hence $X_l^{j_l} \in I$, hence  $X_1^{j_1} \dots X_n^{j_n} \in I$.

\medskip

 $(2) \implies (1)$. 
If  $A=\K[X_1 \dots X_n]/I$ and $I \supset (X_1 \dots X_n)^k$
certainly all the generators $X_1 \dots X_n$ are nilpotents,
hence $A=\K {\oplus} (X_1 \dots X_n)$, where $(X_1 \dots X_n)$ is the
maximal ideal generated by the generators (by abuse of notation we
use the same letter also in the quotient). Clearly this ideal
is not all $A$, since it consists only of nilpotent elements. 
Notice that since the $X_i$ are nilpotent one can readily check
that $f_1X_1+ \dots +f_nX_n$ is nilpotent (choosing as
exponent $nN$, $N=max\{n_i\}$ and reasoning as before).

\medskip

$(1) \implies (3)$. 
By lemma 4.3.2 in \cite{Varadarajan}
pg 140, we have that $\maxid_0$ in $\sheaf_{\K^{p|q},0}$ is
generated by $x_1 \dots x_n$, where $x_i$ are the germs
of local coordinates around $0$. Moreover if $f \in \sheaf_{\K^{p|q},0}$,
for all $K$, there exists a polinomial $P$ in the $x_1 \dots x_n$
such that any $f - P \in \maxid_0^K$ by \ref{lemma:polynomials}.

\medskip

If $A$ is a Weil superalgebra and $a_1 \dots a_n$ its
(nilpotent) homogeneous generators, by the previous discussion we have
that $A=\K[X_1 \dots X_n]/I$, $I \supset (X_1 \dots X_n)^k$
for a suitable $k$ (specified above). 
Choose $x_1 \dots x_n$
as the coordinates of $\K^{p|q}$ ($n=p+q$). Clearly the polinomial
algebra in such coordinates $\K[x_1 \dots x_n]$ embeds
into $\sheaf_{\K^{p|q},0}$ (again, with a small
abuse of notation we use $x_i$ to denote both the germs and
the polynomial coordinates).

\medskip

Let $J:=<I>$ be the
ideal in $\sheaf_{\K^{p|q},0}$ generated by the image of
$I$ in  $\sheaf_{\K^{p|q},0}$.
Notice that $\maxid_0^k \subset J$,
since $(X_1 \dots X_n)^k \subset I$.

\medskip

We claim $A \cong \sheaf_{\K^{p|q},0}/J$.
Certainly we can define a morphism 
$\phi:\K[X_1\dots X_n] \lra \sheaf_{\K^{p|q},0}/J$,
$X_i \mapsto x_i$. This morphism factors through 
$A=\K[X_1 \dots X_n]/I$ since, by the very definition of $J$, 
$\phi(I) \subset J$.
So we have obtained a well defined morphism:
$$
\begin{array}{ccc}
\psi: A=\K[X_1 \dots X_n]/I & \lra & \sheaf_{\K^{p|q},0}/J \\
X_i & \mapsto & x_i
\end{array}
$$
We want to show it is surjective and injective.

\medskip

$\psi$ surjective.
Let $f \in \sheaf_{\K^{p|q},0}/J$. By Hadamard's lemma
we have that there exists a polinomial $P$ in the $x_1 \dots x_n$
such that $f - P \in \maxid_0^k$. 
Since $\maxid_0^k \subset J$,
we have $f=P$ in $\sheaf_{\K^{p|q},0}/J$, hence $\psi$ is surjective.

\medskip
$\psi$ injective. 
Assume $\psi(p)=0$, that is $\psi(p) \in J$. 
We have that
$$
\psi(p)=u_1p_1+ \dots +u_np_n, \quad
u_i \in \sheaf_{\K^{p|q},0}, \quad I=(p_1 \dots p_n) \subset
\K[x_1 \dots x_n] \subset \sheaf_{\K^{p|q},0}
$$
(again we identify $\K[X_1 \dots X_n]$ with the subring
$\K[x_1 \dots x_n]$
of $\sheaf_{\K^{p|q},0}$ and look at $I$ as an ideal inside it).

\medskip

Again by Hadamard's lemma, we have that 
$u_i-P_i \in \maxid_0^K$ for all $K$, $P_i \in \K[x_1 \dots x_n]$, that
is $u_i=P_i+m_i$, with $m_i \in \maxid_0^K$.
So we have:
$$
\psi(p)=u_1p_1+ \dots +u_np_n=\sum_i (P_i+m_i)p_i =
\sum_i P_ip_i+ \sum_i m_ip_i \in I+\maxid_0^K
$$
If we set $P:=\sum_i P_ip_i \in I$, we have that
$$
\psi(p)-P=\sum_i m_ip_i \in \maxid_0^K, \qquad \forall K.
$$
Now we want to show that  if a polynomial $\psi(p)-P$ is
in $\maxid_0^K$, then the polynomial is also in $(x_1 \dots x_n)^K \subset 
\K[x_1 \dots x_n]$ (let's not forget that choosing
the correct $K$ we have $(x_1 \dots x_n)^K \subset I$).
But looking at the formula at pg 141 in \cite{Varadarajan} 
and at the expression of the remainder $R_K(x)$, we see this
is true at once. 

\medskip

$(3) \implies (1)$. Now  assume
$A=\sheaf_{\K^{p|q},0}/J$
{with $J\supset \maxid_0^k$, and hence of finite codimension}.
Since $A$ is local we have $A=\K \oplus \maxid_0$, where $\maxid_0$ is
generated by the nilpotent elements $x_1 \dots x_n$.
By then we are done because every element of $\maxid_0$ is nilpotent,
hence we can reason as above.

\end{proof}

\subsection{$A$\nbd points} \label{subsec:A-points}

We now introduce the notion of $A$-point of a supermanifold $M$.
Despite the analogy with the functor of points described previously,
the functor associated with the $A$-points is subtly different.
The main difference is that the collection of the $A$-points of
a supermanifold $M$, for all super Weil algebras $A$, will not enable
us to recover all the information about the supermanifold $M$.
As we are going to see, in order to transfer all the information
from the supermanifold $M$ to the collection of its $A$-points it
is necessary to endow each of such sets of an extra structure, that
we are going to discuss in sec. 4.

\begin{definition}
Let $M$ be a supermanifold, $x \in \topo{M}$ and $A$ a super Weil
algebra. We define the \emph{set of $A$\nbd points near $x$} as
\[
    M_{A,x} \coloneqq \Hom_\SAlg (\stalk{M}{x}, A)
\]
and the \emph{set of $A$\nbd points} as
\[
    M_A \coloneqq \bigsqcup_{x \in \topo{M}} M_{A,x} \text{.}
\]
If $x_A \in M_{A,x}$, we call $\red{x_A} \coloneqq x$ the \emph{base
point} of $x_A$.
\end{definition}

\begin{observation}
Notice that, since $\stalk{M}{x}$ is a local algebra, $M_{\K,x}$
contains only the evaluation $\ev_x$ and hence $M_\K$ is identified
with the set of topological points of $M$. Moreover, for each $A \in
\SWA$ and $x_A \in M_A$, we have that
\[
    x_A = \ev_{\red{x_A}} + L
\]
where $\im(L) \subseteq \nil{A}$.
\end{observation}

We can consider the functor
\begin{equation} \label{eq:A-point_functor}
    M_{(\blank)} \colon \SWA \to \Sets
\end{equation}
defined on the objects as
\[
    A \mapsto M_A
\]
and on morphisms as $\rho \mapsto \funcpt{\rho}$, with $\rho \in
\Hom_{\SAlg}(A,B)$ and
\[
    \begin{aligned}
        \funcpt{\rho} \colon M_A &\to M_B \\
        x_A &\mapsto \rho \circ x_A \text{.}
    \end{aligned}
\]

\begin{remark} \label{remark:localalg}
Observe that the only local superalgebras which are equal to
$\sheaf(M)$ for some supermanifold $M$ are those of the form
$\ext[\K]{\theta_1,\dots,\theta_q} = \sheaf(\K^{0|q})$. For this reason this
functor is quite different from the functor of points borrowed from
algebraic geometry and detailed in the previous section.
\end{remark}

\begin{notation}
Here we introduce a multiindex notation that we will use in
the following. Let $\set{x_1,\ldots,x_p,\theta_1,\ldots,\theta_q}$
be a system of coordinates. If
\[
    \nu = (\nu_1,\ldots,\nu_p) \in \N^p \text{,}
\]
we define
\[
    x^\nu \coloneqq x_1^{\nu_1} x_2^{\nu_2} \cdots x_p^{\mu_p} \text{,}
\]
$\nu! \coloneqq \prod_i \nu_i!$, and $\abs{\nu} \coloneqq \sum_i
\nu_i$. If $J=(j_1,\ldots,j_r)$ with $1 \leq j_1 < \dots < j_r \leq q$
we define
\[
    \theta^J \coloneqq \theta_{j_1} \theta_{j_2} \cdots \theta_{j_r}.
\]
$\abs{J}$ denotes the cardinality of $J$.
\end{notation}

In order to understand the structure of $M_A$ we need some
preparation. We start with a  well known result that holds for
smooth supermanifolds and for  {Stein} superdomains $U \subseteq
\C^{p|q}$.

\begin{lemma}[``Super" Milnor's exercise]
Denote by $M$ either
\begin{enumerate}
    \item a smooth supermanifold or
    \item a {Stein} superdomain in $\C^{p|q}$.
\end{enumerate}
The superalgebra maps $\sheaf(M) \to \K$ are exactly the evaluations
$\ev_x \colon s \mapsto \red{s}(x)$ in the points $x \in \topo{M}$.
In other words there is a bijective correspondence between
$\Hom_{\SAlg}\big(\sheaf(M),\K\big)$ and $\topo{M}$.
\end{lemma}

\begin{proof}{
In the smooth case the lemma is a consequence of eq.\
\eqref{eq:pullback_of_global_sections}, considering that
$\sheaf(\R^{0|0}) = \R$ and the pullback of a morphism $\phi \colon
\R^{0|0} \to M$ is the evaluation at $\topo{\phi}(\R^0)$.

In the holomorphic case it is a again a consequence of the analogous classical result 
for Stein manifold (see, for example, \cite{Horm}).
Indeed, if 
\[
\psi\colon\Hol(\topo{M})\otimes \Lambda^q \to \C
\]
is an algebra morphism, it uniquely factorizes through the algebra morphism
\[
\topo{\psi}\colon\Hol(\topo{M}) \to \C
\]
and  we have
\[
\psi(f)=\topo{\psi}(\topo{f})
\]
Using  \cite[Proposition 57.1]{Kaup}, we are done.}
\end{proof}

\begin{remark}
Notice that the lemma does not hold
for a generic holomorphic supermanifold.
\end{remark}

Let $\psi \in \Hom_\SAlg \big( \sheaf(M),A \big)$. Due to the
previous lemma, there exists a unique point of $\topo{M}$, that we
denote by $\red{\psi}$, such that $\pr_{A} \circ \psi =
\ev_{\red{\psi}}$, where $\pr_A$ is the projection $A \to \K$.
We thus have a map
\begin{equation} \label{eq:base_point}
    \begin{aligned}
        \Hom_\SAlg \big( \sheaf(M),A \big) &\to \Hom_\SAlg \big( \sheaf(M),\K \big) \isom \topo{M} \\
        \psi &\mapsto \pr_{A} \circ \psi = \ev_{\red{\psi}} \text{.}
    \end{aligned}
\end{equation}



The next proposition establishes the local nature of the functor
$A \mapsto M_A$.

\begin{proposition}
\label{prop:valori_assegnati}
\begin{enumerate}
    \item\label{it:v_ass_germs} Each element $x_A$ of $M_A$ is determined by the images of
    the germs of a system of local coordinates $\germ{x_i},
    \germ{\theta_j}$ around $\red{x_A}$. Conversely, given $x \in
    \topo{M}$, a system of local coordinates $\set{x_i}_{i=1}^p,
    \set{\theta_j}_{j=1}^q$ around $x$ and elements $\set{\X_i}_{i=1}^p,
    \set{\T_j}_{j=1}^q$, $\X_i \in A_0$, $\T_j \in A_1$,\footnote{%
        The reader should notice the difference between
        $\set{x_i,\theta_j}$ and $\set{\X_i,\T_j}$.
    } such that $\red{\X_i} = \red{x_i}(x)$, there
    exists a unique morphism $x_A \in \Hom_\SAlg(\stalk{M}{x},A)$ such
    that
    \begin{equation} \label{eq:valori_assegnati}
        \left\{
        \begin{aligned}
            x_A(\germ{x_i}) & = \X_i \\
            x_A(\germ{\theta_j}) & = \T_j \text{.}
        \end{aligned}
        \right.
    \end{equation}
   {\item\label{it:associated_chart}  Suppose $(U,h)$ is a coordinate chart, then there is a  bijection 
   \begin{align*}
   U_A & \rightarrow \topo{h(U)}\times\nil{A}_0\times \nil{A}_1
   \end{align*}
   }
    \item\label{it:v_ass_chart} {Suppose  $U$ is a coordinate chart  in the smooth case, and a Stein  coordinate chart  in the holomorphic case, then} 
    there is a bijective correspondence
    \[
        U_A = \bigsqcup_{x \in U} \Hom_\SAlg (\stalk{M}{x}, A) \to \Hom_\SAlg \big( \sheaf_M(U),A \big) \text{.}
    \]
\end{enumerate}
\end{proposition}

\begin{proof}
Let us consider \itref{\ref{it:v_ass_germs}}. Suppose that $x_A$ is given. We want to show that the images of the
germs of local coordinates $x_A(\germ{x_i})$, $x_A(\germ{\theta_j})$
determine $x_A$ completely. This follows
noticing that 
\begin{itemize}
    \item the image of a polynomial section under $x_A$ is determined,
    \item there exists $k \in \N$ such that the kernel of $x_A$ contains
    $\maxid_x^k$ (see lemma \ref{lemma:SWA}) 
\end{itemize}
and using Hadamard's lemma. We now come to existence. Suppose the eq.\
\eqref{eq:valori_assegnati} are given and let $\germ{s}$ be a germ
at $x$. We define $x_A(\germ{s})$ through a formal Taylor expansion.
More precisely let
\[
    s = \sum_{J \subseteq \set{1,\ldots,q}} s_J \theta^J
\]
be a representative of $\germ{s}$ near $x$, where the $s_J$ are
smooth (holomorphic) functions in $x_1,\ldots,x_p$. Define
\begin{equation} \label{eq:formaltaylor}
    x_A(s)
    = \sum_{\substack{\nu \in \N^p \\ J \subseteq \set{1,\ldots,q}}}
    \frac{1}{\nu!} \frac{\partial^{\abs{\nu}} s_J}{\partial x^{\nu}} \bigg|_{(\red{\X_1},\ldots,\red{\X_p})} \nil{\X}^\nu \T^J \text{.}
\end{equation}
This is the way in which the purely formal expression
\[
    s(x_A) = s(\red{\X_1}+\nil{\X_1},\dots,\red{\X_p}+\nil{\X_p},\T_1,\dots,\T_q)
\]
is usually understood. Eq.\ \eqref{eq:formaltaylor}  has only a
finite number of terms due to the nilpotency of the $\nil{\X_i}$ and
$\T_j$. It is clear from eq.\ \eqref{eq:formaltaylor} that
$x_{A}(s)$ does not depend on the chosen representative. Finally
$x_{A}$ so defined is a superalgebra morphism since, for each
$\germ{s}, \germ{t} \in \stalk{M}{x}$,
\begin{align*}
    x_A(s t)
    &= \sum_{\substack{\nu \in \N^p \\ K \subseteq J \subseteq \set{1,\ldots,q}}}
    \lambda(K,J) \frac{1}{\nu!} \frac{\partial^{\abs{\nu}} (s_K t_{J \setminus K})}{\partial x^{\nu}} \nil{\X}^\nu \T^J \\
    &= \sum_{\nu-\mu \in \N^p, K, J} \lambda(K,J)
    \frac{1}{\nu!} \binom{\nu}{\mu}
    \frac{\partial^{\abs{\mu}} s_K}{\partial x^{\mu}}
    \frac{\partial^{\abs{\nu-\mu}} t_{J \setminus K}}{\partial x^{\nu-\mu}} \nil{\X}^\nu \T^J \\
    &= \sum_{\nu,\mu,J,K}
    \Bigg[ \frac{1}{\mu!} \frac{\partial^{\abs{\mu}} s_K}{\partial x^{\mu}} \nil{\X}^\mu \T^K \Bigg]
    \Bigg[ \frac{1}{(\nu-\mu)!} \frac{\partial^{\abs{\nu-\mu}} t_{J \setminus K}}{\partial x^{\nu-\mu}} \nil{\X}^{\nu-\mu} \T^{J \setminus K} \Bigg]
\end{align*}
where $\binom{\nu}{\mu} = \prod_i \binom{\nu_i}{\mu_i}$ and
$\lambda(K,J)$ is defined to be  $\pm 1$ according to $\T^K \T^{J \setminus K} =
\lambda(K,J) \T^J$.

\medskip

{\itref{\ref{it:associated_chart}} is just a restatement of \itref{\ref{it:v_ass_germs}}}

\medskip

Let us now consider \itref{\ref{it:v_ass_chart}}. Define the map $\eta_x \colon
\sheaf_M(U) \to \stalk{M}{x}$ assigning to each section over $U$ the
corresponding germ at $x\in U$, and consider
\[
    \begin{aligned}
        \smash[b]{\bigsqcup_{x \in U}} \Hom_\SAlg (\stalk{M}{x}, A) &\to \Hom_\SAlg \big( \sheaf_M(U),A \big) \\
        x_A &\mapsto x_A\circ \eta_{\red{x_A}} \text{.}
    \end{aligned}
\]
We show that it is invertible. Let $\psi \in \Hom_\SAlg \big(
\sheaf_M(U),A \big)$. If $\set{x_i,\theta_j}$ is a coordinate system
on $U$, due to \itref{\ref{it:v_ass_germs}}, the set $\set{\psi(x_i),\psi(\theta_j)}$
uniquely determines an element in $ \Hom_\SAlg
(\stalk{M}{\red{\psi}}, A)$. It is easy to show that this is the
required inverse.
\end{proof}


In the smooth category the above setting can  be somewhat
simplified. This is essentially due to eq.\
\eqref{eq:pullback_of_global_sections} and the related discussion.
Let us see in detail this point, summarized in prop.\
\ref{prop:smoothcase}.

\begin{lemma}
Let $M$ be a smooth supermanifold.
Let $s \in \sheaf(M)$ and let $\psi\in \Hom_\SAlg\big (\sheaf(M), A
\big)$. Assume that $s$ is zero when restricted to a certain
neighbourhood of $\red{\psi}$ (see eq.\ \eqref{eq:base_point}). Then
$\psi(s)=0$.
\end{lemma}

\begin{proof}
Suppose $U \ni \red{\psi}$ is such that $\restr{s}{U}=0$. Let $t \in
\sheaf_M(U)$ be such that $\supp(t) \subset U$ and $\restr{t}{V} =
1$, where the closure of $V$ is contained in $U$. Then
\[
    0 = \psi(st) = \psi(s)\psi(t) \text{.}
\]
Hence $\psi(s)=0$, since $\psi(t)$ is invertible being
$\ev_{\red{\psi}}(t) = 1$.
\end{proof}

\begin{proposition} \label{prop:smoothcase}
If $M$ is a smooth supermanifold
\begin{equation*}
    M_A \isom \Hom_{\SAlg_\R} \big( \sheaf(M), A \big)
\end{equation*}
in a functorial way. In other words we can equivalently define the
functor $M_{(\blank)}$ on the objects as:
$$
\begin{array}{cccc}
M_{(\blank)} \colon  & \SWA & \to & \Sets \\ \\
& A & \mapsto &  \Hom_{\SAlg_\R} \big( \sheaf(M), A \big).
\end{array}
$$

\end{proposition}

\begin{proof}
Clearly each $x_A \in M_A$ can be identified with a superalgebra map
$\sheaf(M) \to A$ by composing it with the natural map $\eta_x
\colon \sheaf(M)\to \stalk{M}{x}$. Vice versa let $\psi \in
\Hom_{\SAlg_\R} \big( \sheaf(M), A \big)$. In the smooth category, given
a germ $\germ{s} \in \stalk{M}{\red{\psi}}$, there exists a global
section $s \in \sheaf(M)$ such that $\eta_{\red{\psi}}(s) =
\germ{s}$. Since the image of $s$ under $\psi$ depends only on the
germs of $s$ at $\red{\psi}$, $\psi$ determines an element of
$M_{A,x}$. In fact, let $s' \in \germ{s}$ and let $U$ be a
neighbourhood of $\red{\psi}$ such that $\restr{s'}{U} =
\restr{s}{U}$. It is always possible to find a smaller neighbourhood
$V \subset U$ and $u,v_1,v_2 \in \sheaf(M)$ such that $s = u + v_1$,
$s' = u + v_2$ and $\restr{v_i}{V} = 0$. Then, due to the previous
lemma, $\psi(s) = \psi(u) = \psi(s')$. The functoriality is clear.
\end{proof}

The next observation gives us an interesting and very important
characterization of the $A$-points of a superdomain.

\begin{observation} \label{obs:coordinates}
Let $(U,h)$ be a chart in a supermanifold $M$ with local
coordinates $\set{x_i,\theta_j}$. By point \itref{2} of prop.\ \ref{prop:valori_assegnati}
we have an injective map
\[
    \begin{aligned}
        U_A &\to A_0^p \times A_1^q \\
        x_A &\mapsto (\X_1,\ldots,\X_p,\T_1,\ldots,\T_q) \coloneqq
            \big( x_A(x_1),\ldots, x_A(\theta_q) \big) \text{.}
    \end{aligned}
\]
We can think of it
heuristically as the assignment of $A$\nbd valued
coordinates $\set{\X_i,\T_j}$ on $U_A$. As we are going to see in theorem
\ref{theor:azerolinear} the components of the coordinates
$\{\X_i, \T_j\}$, given by
$\pair{\smash{a^*_k}}{\X_i}$, $\pair{\smash{a^*_k}}{\T_j}$
with respect to a basis $\set{a_k}$ of $A$ are indeed the
coordinates of a smooth or holomorphic manifold. The base point
$\red{x_A} \in U$ has coordinates $(\red{\X_1},\ldots,\red{\X_p})$.
In this language, if $\rho \colon A \to B$ is a super Weil algebra
morphism, the corresponding morphism $\funcpt{\rho} \colon M_A \to
M_B$ is ``locally'' given by
{
\begin{equation} \label{eq:rhoX...Xrho}
    \rho \times \dots \times \rho \colon \topo{U}\times\nil{A}_0^p \times A_1^q \to \topo{U}\times\nil{B}_0^p \times B_1^q \text{.}
\end{equation}
where, with an harmless abuse of notation, we are confusing $\topo{h(U)}$ with $\topo{U}$.}
This is well defined since $\funcpt{\rho}$ does not change the base
point.

If $M = \K^{p|q}$ we can also consider the slightly different
identification
\[
    \begin{aligned}
        \K^{p|q}_A &\to (A \otimes \K^{p|q})_0 \\
        x_A &\mapsto \sum_i x_A(e_i^*) \otimes e_i
    \end{aligned}
\]
where $\set{e_1,\ldots,e_{p+q}}$ denotes a homogeneous basis of
$\K^{p|q}$ and $\set{e_1^*,\ldots,e_{p+q}^*}$ its dual basis. Here a
little care is needed as we already remarked at the beginning of
section \ref{sec:basic_def}.
In the literature the name $\K^{p|q}$ is used
for two in principle different objects: it may indicate the super
vector space $\K^{p|q}=\K^p \oplus \K^q$ or the superdomain
{$(\K^p,\sheaf_{\K^p} \otimes \extn{q})$}. In the previous equation the
first $\K^{p|q}$ is viewed as a superdomain, while the last as a
super vector space. Likewise the $\set{e_i^*}$ are
interpreted both as
vectors and sections of $\sheaf(\K^{p|q})$. As we shall see in
subsection \ref{subsec:A0-smooth}
the functor
\[
    A \mapsto (A \otimes \K^{p|q})_0
\]
recaptures all the information about the superdomain $\K^{p|q}$,
so that the two in principle different ways of
looking at $\K^{p|q}$ become then identified.
This result will hence establish a quite natural
way to identify the two objects. With this identification,
the superdomain morphism
$\funcpt{\rho} \colon \K^{p|q}_A \to \K^{p|q}_B$
corresponds to the super vector space morphism
\[
    \rho \otimes \id \colon (A \otimes \K^{p|q})_0 \to (B \otimes \K^{p|q})_0 \text{.}
\]
\end{observation}

\subsection{Natural transformations between functors of $A$\nbd points}
\label{subsec:A-point_nat_trans}

In the previous subsection we have seen that, somehow
mimicking the functor of points approach to supermanifolds, it is
possible to associate to each supermanifold $M$ a functor
\[
    \begin{aligned}
        \SWA &\to \Sets \\
        A &\mapsto M_A \text{.}
    \end{aligned}
\]
Hence we have a functor:
\begin{equation*}
    \ber \colon \SMan \to \funct{\SWA}{\Sets} \text{.}
\end{equation*}
The natural question about such a  functor is whether $\ber$ is a full
and faithful embedding or not. In this subsection, we show that
$\ber$ is not full, in other words, there are many more natural
transformations between $M_{(\blank)}$ and $N_{(\blank)}$ than those
coming from morphisms from $M$ to $N$. We will show this by giving a
simple example. Then, in prop.\ \ref{prop:formal_series}, we will
see a characterization of the natural transformation between
two superdomains.

\medskip

Let us start our discussion. We first want to
show that the natural transformations $M_{(\blank)} \to
N_{(\blank)}$ arising from supermanifold morphisms $M\to N$ have a
very peculiar form. Indeed, a morphism $\phi \colon M \to N$
of supermanifolds induces a natural
transformation between the corresponding functors of $A$\nbd points
given by
\[
    \begin{aligned}
        \phi_A \colon M_A &\to N_A \\
        x_A &\mapsto x_A \circ \phi^*
    \end{aligned}
\]
for all super Weil algebras $A$. Let $M =
\K^{p|q}$ and $N = \K^{m|n}$, and denote respectively by
$\set{x_i,\theta_j}$ and $\set{x'_k,\theta'_l}$ two systems of
coordinates over them. With these assumptions, $\phi$ is determined
by the pullbacks of the coordinates of $N$, while the $A$\nbd point
$\phi_A(x_A)$ is determined by
\[
    (\X'_1,\ldots,\X'_{m},\T'_1,\ldots,\T'_{n}) \coloneqq
    \big( x_A \circ \phi^*(x'_1),\ldots,x_A \circ \phi^*(\theta'_{n}) \big) \in A_0^{m} \times A_1^{n} \text{.}
\]
If $(\X_1,\ldots,\X_p,\T_1,\ldots,\T_q)$ denote the images of the
coordinates of $M$ under $x_A$ ($\X_1 = x_A(x_1)$, etc.) and
$\phi^*(x'_k) = \sum_J s_{k,J} \theta^J \in \sheaf(\K^{p|q})_0$,
where the $s_{k,J}$ are functions on $\K^p$, then we have
\begin{equation} \label{eq:nat_tr_from_morph}
    \X'_k = x_A \circ \phi^*(x'_k)
    = \sum_{\substack{\nu \in \N^p \\ J \subseteq \set{1,\ldots,q}}}
    \frac{1}{\nu!} \frac{\partial^{\abs{\nu}} s_{k,J}}{\partial x^{\nu}} \bigg|_{(\red{\X_1},\ldots,\red{\X_p})} \nil{\X}^\nu \T^J
\end{equation}
and similarly for the odd coordinates (see prop.\ \ref{prop:valori_assegnati}).
Notice that if we pursue
the point of view of observation \ref{obs:coordinates}, i.~e.\ if we
consider $\set{\X_i,\T_j}$ as $A$\nbd valued coordinates of
$\K^{p|q}_A$, this equation can be read as a coordinate expression
for $\phi_A$.

Not all the natural transformations
$M_{(\blank)} \to N_{(\blank)}$ arise in this way. This happens also
for purely even manifolds, as we see in the next example.

\begin{example} \label{exampla:counter-example_nat_tr}
Let $M$ and $N$ be two smooth manifolds and let $\phi \colon M \to
N$ be a map (smooth or not). The natural transformation
$\alpha_{(\blank)} \colon M_{(\blank)} \to N_{(\blank)}$
\[
    \begin{aligned}
        \alpha_A \colon M_A &\to N_A \\
        x_A &\mapsto \ev_{\phi(\red{x_A})}
    \end{aligned}
\]
is not of the form seen above, even if $\phi$ is assumed to be
smooth, while we still have $\phi=\alpha_\K$.
\end{example}

We end this subsection with a technical result, essentially due to
A. A. Voronov (see \cite{Voronov}), characterizing all possible
natural transformations between the functors of $A$\nbd points of
two superdomains, hence comprehending also those not arising
from supermanifold morphisms.

\begin{definition}
Let $U$ be an open subset of $\K^p$. We denote by
$\fseries{p}{q}(U)$ the unital commutative superalgebra of formal
series with $p$ even and $q$ odd generators and coefficients in the
algebra $\functions(U,\K)$ of arbitrary functions on $U$, i.~e.\
\[
    \fseries{p}{q}(U) \coloneqq \functions(U,\K)[[X_1,\dots,X_p, \Theta_1,\dots,\Theta_q]] \text{.}
\]
An element $F \in
\fseries{p}{q}(U)$ is of the form
\[
    F
    = \sum_{\substack{\nu \in {\N}^p \\ J  \subseteq \set{1,\ldots,q}}}
    f_{\nu,J} X^\nu \Theta^J
\]
where $f_{\nu,J} \in \functions(U,\K)$ and $\set{X_i}$ and
$\set{\Theta_j}$ are even and odd generators. $\fseries{p}{q}(U)$ is
a graded algebra: $F$ is even (resp.\ odd) if $\abs{J}$ is even
(resp.\ odd) for each term of the sum.
\end{definition}

Let us introduce a partial order between super Weil algebras by
saying that $A' \preceq A$ if and only if $A'$ is a quotient of $A$.

\begin{lemma} \label{lemma:upper_bound}
The set of super Weil algebras is directed, i.~e., if $A_1$ and
$A_2$ are super Weil algebras, then there exists $A$ such that $A_i
\preceq A$.
\end{lemma}

\begin{proof}
In view of lemma \ref{lemma:SWA}, choosing carefully $k, l \in \N$
and $J_1$ and $J_2$ ideals of $\kspoly{k}{l}$, we have $A_i \isom
\kspoly{k}{l} / J_i$. If $r$ is the maximum between the heights of
$A_1$ and $A_2$, $\maxid^{r+1} \subset J_1 \cap J_2$. So $A \isom
\kspoly{k}{l} / (J_1 \cap J_2)$ is finite dimensional and then it is
a super Weil algebra.
\end{proof}

\begin{proposition} \label{prop:formal_series}
Let $U$ and $V$ be two superdomains
in $\K^{p|q}$ and $\K^{m|n}$
respectively. The set of natural transformations in
$\funct{\SWA}{\Sets}$ between $U_{(\blank)}$ and $V_{(\blank)}$ is
in bijective correspondence with the set of elements of the form
\[
    \F = (F_1,\ldots,F_{m+n}) \in \big( \fseries{p}{q}(\topo{U}) \big)_0^m \times \big( \fseries{p}{q}(\topo{U}) \big)_1^n
\]
such that, if $F_k = \sum_{\nu,J} f^k_{\nu,J} X^\nu \Theta^J$,
\begin{equation} \label{eq:restricions}
    \big( f^1_{0,\emptyset}(x),\ldots,f^m_{0,\emptyset}(x) \big) \subseteq \topo{V}
    \qquad \forall x \in \topo{U} \text{.}
\end{equation}
\end{proposition}

\begin{proof}
As above, $\K^{p|q}_A$ is identified with $A_0^p \times A_1^q$ and
consequently a map $\K^{p|q}_A \to \K^{m|n}_A$ consists of
a list of $m$ maps $A_0^p
\times A_1^q \to A_0$ and $n$ maps $A_0^p \times A_1^q \to A_1$. In
the same way, $U_A$ is identified with $\topo{U} \times \nil{A_0}^p
\times A_1^q$.

Let $\F = (F_1,\ldots,F_{m+n})$ be as in the hypothesis. A formal
series $F_k$ determines a map $\topo{U} \times \nil{A_0}^p \times
A_1^q \subseteq A_0^p \times A_1^q \to A$ in a natural way, defining
\[
    F_k(\X_1,\ldots,\X_p,\T_1,\ldots,\T_q)
    \coloneqq \sum_{\substack{\nu \in \N^p \\ J \subseteq \set{1,\ldots,q}}}
    f^k_{\nu,J}(\red{\X_1},\ldots,\red{\X_p}) \nil{\X}^\nu \T^J \text{.}
\]
The parity of its image is the same as that of $F_k$. Then, in view
of the restrictions imposed on the first $m$ $F_k$ given by eq.\
\eqref{eq:restricions}, $\F$ determines a map $U_A \to V_A$ and,
varying $A \in \SWA$, a natural transformation $U_{(\blank)} \to
V_{(\blank)}$, as it is easily checked.

Let us now suppose now that $\alpha_{(\blank)} \colon U_{(\blank)} \to
V_{(\blank)}$ is a natural transformation. We will see that it is
determined by a unique $\F$ in the way just explained.

Let $A$ be a super Weil algebra of height $r$ and
\[
    x_A = (\red{\X_1} + \nil{\X_1},\ldots,\red{\X_p} + \nil{\X_p},\T_1,\ldots,\T_q)
    \in A_0^p \times A_1^q \isom \K^{p|q}_A
\]
with $\red{x_A} \in \topo{U}$. Let us consider the super Weil
algebra
\begin{equation} \label{eq:hatA}
    \hat{A} \coloneqq \big( \K[z_1,\ldots,z_p] \otimes \ext{\zeta_1,\ldots,\zeta_q} \big) / \maxid^{s}
\end{equation}
with $s > r$ ($\maxid$ is as usual the maximal ideal of polynomials
without constant term) and the $\hat{A}$\nbd point
\[
    y_{\red{x_A}} \coloneqq (\red{\X_1} + z_1,\ldots,\red{\X_1} + z_p,\zeta_1,\ldots,\zeta_q)
    \in \hat{A}_0^p \times \hat{A}_1^q \isom \K^{p|q}_{\hat{A}} \text{.}
\]

A homomorphism between two super Weil algebras is clearly fixed
by the images of a set of generators, but this assignment must
be compatible with the relations between the generators. The
following assignment is possible due to the definition
of $\hat{A}$.
If $\rho_{x_A} \colon \hat{A} \to A$ denotes the map
\[
    \left\{
    \begin{aligned}
        z_i &\mapsto \nil{\X_i} \\
        \zeta_j &\mapsto \T_j \text{,}
    \end{aligned}
    \right.
\]
then clearly $\funcpt{\rho_{x_A}}(y_{\red{x_A}}) = x_A$.

Let ${(\alpha_{\hat{A}})}_k$ with  $1 \leq k \leq m+n$ be a
component of $\alpha_{\hat{A}}$, and let
\[
    {(\alpha_{\hat{A}})}_k(y_{\red{x_A}}) = \sum_{\nu,J} a^k_{\nu,J}(\red{x_A}) z^\nu \zeta^J
\]
with $a^k_{\nu,J}(\red{x_A}) \in \K$ and $\big(
a^1_{0,\emptyset}(\red{x_A}),\ldots,a^m_{0,\emptyset}(\red{x_A})
\big) \in \topo{V}$; the sum is on $\abs{J}$ even (resp.\ odd), if
$k \leq m$ (resp.\ $k
> m$). Due to the functoriality of $\alpha_{(\blank)}$
\[
    {(\alpha_A)}_k (x_A)
    = {(\alpha_{A})}_k \circ \funcpt{\rho_{x_A}} (y_{\red{x_A}})
    = \rho_{x_A} \circ {(\alpha_{\hat{A}})}_k (y_{\red{x_A}})
    = \sum_{\nu,J} a^k_{\nu,J}(\red{x_A}) \nil{\X}^\nu \T^J \text{,}
\]
so that there exists a non unique $\F$ such that $\F(x_A) =
\alpha_A(x_A)$. Moreover $\F(x_{A'}) = \alpha_{A'}(x_{A'})$ for each
$A' \preceq A$ and $x_{A'} \in U_{A'}$ (it is sufficient to use the
projection $A \to A'$).

If, by contradiction, 
$\F'$ is another list of formal series with this property, there
exists a super Weil algebra $A''$ such that $\F(x_{A''}) \neq
\F'(x_{A''})$ for some $x_{A''} \in U_{A''}$. Indeed if a component
$F_k$ differs in $f^k_{\nu,J}$, it is sufficient to consider $A''
\coloneqq \kspoly{p}{q} / \maxid^{s}$ with $s > \max \big(
\abs{\nu},q \big)$.
\end{proof}

\section{The Weil--Berezin functor and the Shvarts embedding}
\label{sec:WB_functor-Sh_emb}

In the previous section we saw that the functor
\[
    \ber \colon \SMan \to \funct{\SWA}{\Sets}
\]
does not define a full and faithful embedding of $\SMan$ in
$\funct{\SWA}{\Sets}$, the category of functors from
$\SWA$ to $\Sets$. 
Roughly speaking, the root of such a
difficulty can be traced to the fact that the functor
$\ber(M) \colon \SWA \to \Sets$
looks only to the local structure of the supermanifold $M$, hence it loses
all the global information.
For the functor of points as we described it in
\ref{subsec:functorofpoints}, we obtain a full and faithfull embedding
thanks to the Yoneda's lemma. If we try to reproduce its proof
in this different setting, we see that the main obstacle is that
$\id_M$ can no longer be seen as an $M$\nbd point of the
supermanifold $M$ itself.
The following heuristic argument gives a hint on how such
a difficulty can be overcome.

It is well known (see, for example, \cite[\S~1.7]{DM}) that given
graded vector spaces $V=V_0\oplus V_1$ and $W=W_0\oplus W_1$, there
is a bijective correspondence between graded linear maps $V\to W$
and functorial families of $\Lambda_0$\nbd linear maps between
$(\Lambda \otimes V)_0$ and $(\Lambda \otimes W)_0$, for each
Grassmann algebra $\Lambda$. This result goes under the name of
\emph{even rule principle}. Since vector spaces are local models for
manifolds, the even rule principle seems to suggest that each $M_A$
should be endowed with a local structure of $A_0$\nbd module. This
vague idea is made precise with the introduction of the category
$\cAoMan$ of $A_0$\nbd smooth manifolds. We then prove that each
$M_A$ can be endowed in a canonical way with the structure of
$A_0$\nbd manifold. This construction allows to specialize the
arrival category of the functor of $A$\nbd points associated to a
supermanifold $M$ and to define the Weil--Berezin functor of $M$ as
\[
    \begin{aligned}
        M_{(\blank)} \colon \SWA &\to \cAoMan \\
        A &\mapsto M_A \text{.}
    \end{aligned}
\]
In this way we can define a functor
\[
    \begin{aligned}
        \sch \colon \SMan &\to \dfunct{\SWA}{\cAoMan}\\
        M &\mapsto M_{(\blank)}
    \end{aligned}
\]
where $\dfunct{\SWA}{\cAoMan}$ is an appropriate subcategory of
$\funct{\SWA}{\cAoMan}$ that will be specified by definition
\ref{def:Az_nat_transf}. We will call $\sch$ the Shvarts embedding.

As it turns out, this definition of the local functor of points
allows to recover the correct natural transformations without any
artificial condition. More precisely, lemma
\ref{lemma:A0_smooth_nat_tr} will show that the natural
transformations $\alpha_{(\blank)} \colon M_{(\blank)} \to
N_{(\blank)}$ arising from supermanifolds morphisms are exactly
those for which $\alpha_A$ is $A_0$\nbd smooth for each $A$. $\sch$
is then a full and faithful embedding. Moreover $\sch$ preserves the
products. In particular this implies that if $G$ is a group object
in $\SMan$, i.~e.\ it is a super Lie group, then $\sch(G)$ is a
group object too, i.~e.\ it takes values in the category of the
$A_0$\nbd smooth Lie groups.

Exactly as for the functor of points, the functor $\sch$ is not an
equivalence of categories, so that the problem of characterizing the
functors in the image of $\sch$ arises naturally (representability
problem). A criterion characterizing the representable functors is
then given in subsection \ref{subsec:representability}.

\subsection{$A_0$\nbd smooth structure and its consequences}
\label{subsec:A0-smooth}

Preliminary to everything is the following (rather long) definition
of $A_0$\nbd manifold and of the category $\cAoMan$.
For a more detailed discussion see for example \cite{Shurygin} and
references therein.

\begin{definition} \label{def:Az_man}
Fix an even commutative finite dimensional algebra $A_0$ and let $L$
be a finite dimensional $A_0$\nbd module. Let $M$ be a manifold. An
\emph{$L$\nbd chart} on $M$ is a pair $(U,h)$ where $U$ is open in
$M$ and $h \colon U \to L$ is a diffeomorphism onto its image. $M$
is an \emph{$A_0$\nbd manifold} if it admits an $L$\nbd atlas. By this we
mean a family $\set{(U_i,h_i)}_{i \in \cA}$ where $\set{U_i}$ is an
open covering of $M$ and each $(U_i,h_i)$ is an $L$\nbd chart, such
that the differentials
\[
    d(h_i \circ h_j^{-1})_{h_j(x)} \colon T_{h_j(x)}L \isom L \to L \isom T_{h_i(x)}L
\]
are isomorphisms of $A_0$\nbd modules for all $i$, $j$ and $x \in
U_i \cap U_j$.

If $M$ and $N$ are $A_0$\nbd manifolds, a \emph{morphism} $\phi \colon M
\to N$ is a smooth map whose differential is $A_0$\nbd linear at
each point. We also say that such morphism is \emph{$A_0$\nbd smooth}.
We denote by $\AoMan$ the category of $A_0$\nbd
manifolds.

We define also the category  $\cAoMan$ in the following way. The
objects of  $\cAoMan$ are manifolds over generic finite dimensional
commutative algebras. The morphisms in the category are defined as
follows. Denote by $A_0$ and $B_0$  two commutative finite
dimensional algebras, and let $\rho \colon A_0\to B_0$ be  an
algebra morphism. Suppose  $M$ and $N$ are $A_0$ and $B_0$ manifolds
respectively, we say that a morphism $\phi \colon M \to N$ is
\emph{$\rho$\nbd smooth} if $\phi$ is smooth and
\[
    (d\phi)_x(a v) = \rho(a)(d\phi)_x(v)
\]
for each $x \in M$, $v \in T_x(M)$, and $a \in A_0$.
\end{definition}

Notice that $A_0$\nbd linearity always implies $\K$\nbd linearity,
in particular, in the complex case, $A_0$\nbd manifolds are
holomorphic.

The above definition is motivated by the following theorems.
In order to ease the exposition we first give the statements of
the results and we postpone their proofs to the last part of this
subsection.

\begin{theorem} \label{theor:azerolinear}
Let $M$ be a smooth (resp.\ holomorphic) supermanifold, and let $A$
be a real (resp.\ complex) super Weil algebra.
\begin{enumerate}
    \item $M_A$ can be endowed with a unique $A_0$\nbd manifold structure such
    that, for each open subsupermanifold $U$ of $M$ and $s \in \sheaf_M(U)$
    the map defined by
    \[
        \begin{aligned}
            \hat{s} \colon U_A &\to A \\
            x_A &\mapsto x_A(s)
        \end{aligned}
    \]
    is $A_0$\nbd smooth.

    \item If $\phi \colon M\to N$ is a supermanifold
    morphism, then
    \[
        \begin{aligned}
            \phi_A \colon M_A &\to N_A \\
            x_A &\mapsto x_A \circ \phi^*
        \end{aligned}
    \]
    is an $A_0$\nbd smooth morphism.

    \item If $B$ is another super Weil algebra and $\rho \colon A\to B$ is an
    algebra morphism, then
    \[
        \begin{aligned}
            \funcpt{\rho} \colon M_A &\to M_B \\
            x_A &\mapsto \rho \circ x_A
        \end{aligned}
    \]
    is a $\restr{\rho}{A_0}$\nbd smooth map.
\end{enumerate}
\end{theorem}

The above theorem says that supermanifolds morphisms give rise to
morphisms in the $\AoMan$ category.  From this point of view the next
definition is quite natural.

\begin{definition} \label{def:Az_nat_transf}
We call $\dfunct{\SWA}{\cAoMan}$ the subcategory of
$\funct{\SWA}{\cAoMan}$ whose objects are the same and whose
morphisms $\alpha_{(\blank)}$ are the natural transformations $\cF
\to \cG$, with $\cF, \cG \colon \SWA \to \cAoMan$, such that
\[
    \alpha_A \colon \cF(A) \to \cG(A)
\]
is $A_0$\nbd smooth for each $A\in \SWA$.
\end{definition}

Theorem \ref{theor:azerolinear}
allows us to give more structure to the arrival category of
the functor of $A$\nbd points. More precisely we have  the following
definition, which is the central definition in our treatment of
the local functor of points.

\begin{definition}
Let $M$ be a supermanifold. We define the \emph{Weil--Berezin
functor} of $M$ as
\begin{equation} \label{eq:WB_functor}
    \begin{aligned}
        M_{(\blank)} \colon \SWA &\to \cAoMan \\
        A &\mapsto M_A \text{.}
    \end{aligned}
\end{equation}
Moreover we define the \emph{Shvarts embedding}
\[
    \begin{aligned}
        \sch \colon \SMan &\to \dfunct{\SWA}{\cAoMan} \\
        M &\mapsto M_{(\blank)} \text{.}
    \end{aligned}
\]
\end{definition}

We can now state one of the main results in this paper; it tells that
the Weil-Berezin functor $M_{(\blank)}$
recaptures all the information contained in
the supermanifold $M$.

\begin{theorem} \label{theor:full_and_faithful}
$\sch$ is a full and faithful embedding, i.~e.\ if $M$ and $N$ are
two supermanifolds, and $M_{(\blank)}$ and $N_{(\blank)}$ their
Weil--Berezin functors, then
\[
    \Hom_\SMan(M,N) \isom \Hom_{ \dfunct{\SWA}{\cAoMan}}(M_{(\blank)},N_{(\blank)}) \text{.}
\]
\end{theorem}

\begin{corollary} 
Two supermanifolds are isomorphic if and only if
their Weil-Berezin functors are isomorphic.
\end{corollary}

\begin{observation}
If we considered the bigger category $\funct{\SWA}{\cAoMan}$ instead
of $\dfunct{\SWA}{\cAoMan}$, the above theorem is no longer true. In
example \ref{exampla:counter-example_nat_tr} we examined a natural
transformation between functors from $\SWA$ to $\Sets$, which did not
come from a supermanifolds morphism.
If, in the same example, $\phi$ is chosen to be smooth,
we obtain a morphism in $\funct{\SWA}{\cAoMan}$ that is not in
$\dfunct{\SWA}{\cAoMan}$. Indeed, it is not difficult to check that
if $\pi_A \colon A \to A$ is given by $a \mapsto \red{a}$, then
$\alpha_A$ (in the example) is $\pi_{A_0}$\nbd linear.
\end{observation}

We now examine the proofs of theorems \ref{theor:azerolinear} and
\ref{theor:full_and_faithful}. First we need to
prove theorem \ref{theor:full_and_faithful} in the case of two
superdomains $U$ and $V$ in $\K^{p|q}$ and $\K^{m|n}$ respectively
(lemma \ref{lemma:A0_smooth_nat_tr}).
As usual, if $A$ is a super Weil algebra, $U_A$ and $V_A$ are
identified with $\topo{U} \times \nil{A_0}^p \times A_1^q$ and
$\topo{V} \times \nil{A_0}^{m} \times A_1^{n}$ (see observation
\ref{obs:coordinates}). Then they have a natural structure of
open subsets of $A_0$\nbd modules. The next lemma is due to A.
A. Voronov in \cite{Voronov} and it is the local version of
theorem \ref{theor:full_and_faithful}.

\begin{lemma} \label{lemma:A0_smooth_nat_tr}
A natural transformation $\alpha_{(\blank)} \colon U_{(\blank)} \to
V_{(\blank)}$ comes from a supermanifold morphism $U \to V$ if and
only if $\alpha_A \colon U_A \to V_A$
is $A_0$\nbd smooth for each $A$.
\end{lemma}

\begin{proof}
Due to prop.\ \ref{prop:formal_series} we know that
$\alpha_{(\blank)}$ is determined by $m$ even and $n$ odd formal
series of the form $F_k = \sum_{\nu,J} f^k_{\nu,J} X^\nu \Theta^J$
with $f^k_{\nu,J}$ arbitrary functions in $p$ variables satisfying
eq.\ \eqref{eq:restricions}. Moreover as we have seen
in the discussion before example \ref{exampla:counter-example_nat_tr}
a supermanifold morphism $\phi \colon U \to V$ gives rise to a natural
transformation $\phi_A \colon U_A \to V_A$ whose components are of
the form of eq.\ \eqref{eq:nat_tr_from_morph}.

Let us suppose that $\alpha_A$ is $A_0$\nbd smooth. This clearly
happens if and only if all its components are $A_0$\nbd smooth and
the smoothness request for all $A$ forces all coefficients
$f^k_{\nu,J}$ to be smooth.

Let ${(\alpha_A)}_k$ be the $k$\nbd th component of $\alpha_A$ and
let $i \in \set{1,\ldots,p}$. We want to study
\[
    \begin{aligned}
        \omega \colon A_0 &\to A_j \\
        \X_i &\mapsto {(\alpha_A)}_k(\X_1,\ldots,\X_i,\ldots,\X_p,\T_1,\ldots,\T_q) \text{,}
    \end{aligned}
\]
supposing the other coordinates fixed ($j = 0$ if $1 \leq k \leq p$
or $j = 1$ if $p < k \leq p+q$). Since $\nil{\X_i} \in A_0$ commutes
with all elements of $A$,
\begin{equation} \label{eq:omega}
    \omega(\X_i) = \sum_{t \geq 0} a_t(\red{\X_i}) \nil{\X_i}^t
\end{equation}
with
\begin{equation} \label{eq:a_t}
    a_t(\red{\X_i}) \coloneqq \sum_{\substack{\nu,J \\ \nu_i=t}} f^k_{\nu,J}(\red{\X_1},\ldots,\red{\X_i},\ldots,\red{\X_p}) \nil{\X}^{(\nu-t\delta_i)} \T^J
\end{equation}
($t\delta_i$ is the element of $\N^p$ with $t$ at the $i$\nbd th
component and $0$ elsewhere).

If $\Y = \red{\Y} + \nil{\Y} \in A_0$ and $\omega$ is $A_0$\nbd
smooth
\begin{equation} \label{eq:Alin_diff_1}
    \omega(\X_i + \Y) - \omega(\X_i) = d\omega_{\X_i}(\Y) + o(\Y) = (\red{\Y} + \nil{\Y}) d\omega_{\X_i}(1_A) + o(\Y)
\end{equation}
(where $1_A$ is the unit of $A$). On the other hand, from eq.\
\eqref{eq:omega} and defining
\begin{equation} \label{eq:a'_t}
    a_t'(\red{\X_i}) \coloneqq \sum_{\substack{\nu,J \\ \nu_i=t}}
    \partial_i f^k_{\nu,J}(\red{\X_1},\ldots,\red{\X_i},\ldots,\red{\X_p}) \nil{\X}^{(\nu-t\delta_i)} \T^J
\end{equation}
($\partial_i$ denotes the partial derivative respect to the $i$\nbd th
variable), we have
\begin{equation} \label{eq:Alin_diff_2}
    \begin{aligned}
        \omega(\X_i + \Y) - \omega(\X_i)
        &= \sum_{t \geq 0} a_t (\red{\X_i} + \red{\Y}) (\nil{\X_i} + \nil{\Y})^t - \sum_{t \geq 0} a_t(\red{\X_i}) \nil{\X_i}^t \\
        &= \sum_{t \geq 0} \left( a_t'(\red{\X_i}) \red{\Y} \nil{\X_i}^t +
        a_t(\red{\X_i}) t \nil{\X_i}^{t-1} \nil{\Y} + o(\Y) \right) \\
        &= \red{\Y} \sum_{t \geq 0} a_t'(\red{\X_i}) \nil{\X_i}^t + \nil{\Y} \sum_{t \geq 0} (t+1) a_{t+1}(\red{\X_i}) \nil{\X_i}^t + o(\Y) \text{.}
    \end{aligned}
\end{equation}
Thus, comparing eq.\ \eqref{eq:Alin_diff_1} and
\eqref{eq:Alin_diff_2}, we get that the identity
\[
    (\red{\Y} + \nil{\Y}) d\omega_{\X_i}(1_A)
    = \red{\Y} \sum_{t \geq 0} a_t'(\red{\X_i}) \nil{\X_i}^t + \nil{\Y} \sum_{t \geq 0} (t+1) a_{t+1}(\red{\X_i}) \nil{\X_i}^t
\]
must hold and, consequently, also the following relations must be
satisfied:
\begin{align*}
    \sum_{t \geq 0} a_t'(\red{\X_i}) \nil{\X_i}^t
    &= \sum_{t \geq 0} (t+1) a_{t+1}(\red{\X_i}) \nil{\X_i}^t \\
\intertext{and then, from eq.\ \eqref{eq:a_t} and \eqref{eq:a'_t},}
    \sum_{\nu,J} \partial_i f^k_{\nu,J}(\red{\X_1},\ldots,\red{\X_p}) \nil{\X}^\nu \T^J
    &= \sum_{\nu,J} (\nu_i + 1) f^k_{\nu + \delta_i,J}(\red{\X_1},\ldots,\red{\X_p}) \nil{\X}^\nu \T^J
    \text{.}
\end{align*}

Let us fix $\nu \in \N^p$ and $J \subseteq \set{1,\ldots,q}$. If $A
= \kspoly{p}{q}/\maxid^s$ with $s > \max(\abs{\nu}+1,q)$ ($\maxid$
is as usual the maximal ideal of polynomials without constant term),
we note that necessarily, due to the arbitrariness of
$(\X_1,\ldots,\T_q)$,
\[
    \partial_i f^k_{\nu,J} = (\nu_i + 1) f^k_{\nu + \delta_i,J}
\]
and, by recursion, ${(\alpha_A)}_k$ is of the form of
\eqref{eq:nat_tr_from_morph} with $s_{k,J} = f^k_{0,J}$.

Conversely, let ${(\alpha_A)}_k$ be of the form of eq.\
\eqref{eq:nat_tr_from_morph}. {It is $A_0$\nbd smooth
if and only if it is $A_0$\nbd smooth in each variable. It is
$A_0$\nbd smooth in the even variables for what has been said above
and in the odd variables since it is polynomial in them.}
\end{proof}

In particular the above discussion
shows also that any superdiffeomorphism $U \to U$ gives rise, for each
$A$, to an $A_0$\nbd smooth diffeomorphism $U_A \to U_A$ and then
each $U_A$ admits a canonical structure of $A_0$\nbd manifold.

We now use the results obtained for superdomains in order to prove
theorems \ref{theor:azerolinear} and \ref{theor:full_and_faithful}
in the general supermanifold case. We need to recall the following
elementary result from ordinary differential geometry.

\begin{lemma}
Let $X$ be a set. Suppose a countable covering $\set{U_i}$ and a
collection of injective maps $h_i \colon U_i \to \K^n$ are given,
satisfying the following conditions:
\begin{enumerate}
    \item for each $i$ and $j$, $h_i(U_i \cap U_j)$ is open in $\K^n$;
    \item $h_j \circ h_i^{-1} \colon h_i(U_i \cap U_j) \to h_j(U_i \cap U_j)$ is a diffeomorphism;
    \item for each $x$ and $y$ in $X$, $x \neq y$, there exists $V_x \subseteq U_i$
    and $V_y \subseteq U_j$ such that $x \in V_x$, $y \in V_y$, $V_x \cap V_y = \emptyset$, $h_i(V_x)$ and $h_j(V_y)$ both open.
\end{enumerate}
Then $X$ admits a unique smooth manifold structure such that
$\set{(U_i,h_i)}$ defines an atlas over it.
\end{lemma}

We leave the proof of this lemma to the reader and we return to
the proofs of theorems \ref{theor:azerolinear} and
\ref{theor:full_and_faithful}.

\begin{proof}[Proof of theorem \ref{theor:azerolinear}]
Let $\set{(U_i, h_i)}$ be an atlas over $M$ and $p|q$ the dimension
of $M$. Each chart $(U_i,h_i)$ of such an atlas induces a chart
$\big((U_i)_A, (h_i)_A\big)$, $(U_i)_A = \bigsqcup_{x \in U_i}
M_{A,x}$, over $M_A$ given by
\[
    \begin{aligned}
        (h_i)_A \colon (U_i)_A &\to {U^{p|q}_A} \\
        x_A &\mapsto x_A \circ h_i^* \text{.}
    \end{aligned}
\]
{with $U^{p|q}$ open subset of $\K^{p|q}$.}

The coordinate changes are easily checked to be given, with some
abuse of notation, by $(h_{i} \circ h_{j}^{-1})_A$,
which are $A_0$\nbd smooth due to lemma \ref{lemma:A0_smooth_nat_tr}.
The uniqueness of the $A_0$\nbd manifold structure is
clear. This proves the first point. The other
two points concern only the local behavior of the considered maps
and are clear in view of lemma \ref{lemma:A0_smooth_nat_tr} and eq.\
\eqref{eq:rhoX...Xrho}.
\end{proof}

\begin{proof}[Proof of theorem \ref{theor:full_and_faithful}]
Lemma \ref{lemma:A0_smooth_nat_tr} accounts for the case in
which $M$ and $N$ are superdomains. For the general case, let us
suppose we have
\[
    \alpha \in \Hom_{\dfunct{\SWA}{\cAoMan}} \big( M_{(\blank)} , N_{(\blank)} \big) \text{.}
\]
Fixing a suitable atlas of both supermanifolds, we obtain, in view
of lemma \ref{lemma:A0_smooth_nat_tr},  a family of local morphisms.
Such a family will give a morphism $M \to N$ if and only if they do
not depend on the choice of the coordinates. Let us suppose that $U$
and $V$ are charts on $M$ and $N$ respectively, {$U
\isom U^{p|q}\subseteq \K^{p|q}$, $V \isom V^{m|n}\subseteq \K^{m|n}$}, such that $\alpha_\K(\topo{U})
\subseteq \topo{V}$, and
{\begin{align*}
    h_i \colon U &\to U^{p|q} &
    k_i \colon V &\to V^{m|n} &
    i &= 1,2
\end{align*}}
are two different choices of coordinates on $U$ and $V$ respectively. The natural transformations
\[
    (\hat{\phi}_i)_{(\blank)} \coloneqq \left( k_i \right)_{(\blank)}
        \circ \restr{\left( \alpha_{(\blank)} \right)}{U_{(\blank)}}
        \circ \left( h_i^{-1} \right)_{(\blank)}
     {   \colon U^{p|q}_{(\blank)} \to V^{m|n}_{(\blank)}}
\]
give rise to two morphisms {$\hat{\phi}_i \colon U^{p|q} \to
V^{m|n}$}. If
\[
   \phi_i \coloneqq k_i^{-1} \circ \hat{\phi}_i \circ h_i \colon U \to V \text{,}
\]
we have $\phi_1 = \phi_2$ since $(\phi_i)_{(\blank)} = \restr{\left(
\alpha_{(\blank)} \right)}{U_{(\blank)}}$ and two morphisms that
give rise to the same natural transformation on a superdomain are
clearly equal.
\end{proof}

We end this subsection with the next proposition stating that the
Shvarts embedding preserves products and, in consequence, group
objects.

\begin{proposition} \label{prop:s-emb}
For all supermanifolds $M$ and $N$,
\[
    \sch(M \times N) \isom \sch(M) \times \sch(N) \text{.}
\]
Moreover $\sch(\K^{0|0})$ is a terminal object
in the category $\dfunct{\SWA}{\cAoMan}$.
\end{proposition}

\begin{proof}
The fact that $(M \times N)_A \isom M_A \times N_A$ for all $A$ can
be checked easily. Indeed, let $z_A\in (M\times N)_A$ with
$\red{z_A}=(x,y)$, we have that $\sheaf_x$ and $\sheaf_y$ naturally
inject in $\sheaf_{\red{z_A}}$.
 Hence $z_A$ defines,
by restriction, two $A_0$\nbd points $x_A\in M_A$ and $y_A\in N_A$.
Using prop.\ \ref{prop:valori_assegnati} and rectangular coordinates
over $M\times N$ it is easy to check that such a correspondence is
injective, and is also a natural transformation. Conversely, if
$x_A\in M_{A,x}$ and $y_A\in N_{A,y}$, they define a map $z_A \colon
\sheaf_{x} \otimes \sheaf_{y} \to A$ through $z_A(s_1 \otimes
s_2)=x_A(s_1)\cdot y_A(s_2)$. Using again prop.\
\ref{prop:valori_assegnati}, it is not difficult to check that this
requirement uniquely determines an element in $(M \times
N)_{A,(x,y)}$ and that this correspondence defines an inverse for the
morphism $(M\times N)_{(\blank)}\to M_{(\blank)}\times N_{(\blank)}$
defined above.

Along the same lines it can be proved that, a  similar condition
for the morphisms holds. Finally $\sch(\K^{0|0})$ is a terminal
object, since $\K^{0|0}_A = \K^0$ for all $A$.
\end{proof}

It is easy to check that the stated result is equivalent to the fact
that $\sch$ preserves finite products for arbitrary many objects.

\medskip

We
now consider a super Lie group $G$, i.~e.\ a group object in the
category of (smooth or holomorphic) supermanifolds. First we recall
briefly the notion of group object.

\begin{definition}
A \emph{group object} in some category with finite products and terminal
object $\mathfrak{T}$, is an object $G$ with three arrows
\begin{align*}
    \mu_G \colon G \times G &\to G &
    i_G \colon G &\to G &
    e_G \colon \mathfrak{T} &\to G
\end{align*}
satisfying the usual commutative diagrams for multiplication,
inverse and unit respectively.
\end{definition}

\begin{observation}
In a locally small category $\cat{C}$ we have that,
equivalently, a group  is an object
$G$ whose functor of points $\FOP{G}$ takes value in the category of
groups $\Grp$, i.~e.\ $G$ is a group object if there exists a
functor $\cat{C}\op \to \Grp$  that, composed with the forgetful
functor $\Grp \to \Sets$, equals $\FOP{G}$.
\end{observation}

\begin{corollary}
If $G$ is a super Lie group, $\sch(G)$ with the arrows
$\sch(\mu_G)$, $\sch(i_G)$ and $\sch(e_G)$ is a group object in
$\dfunct{\SWA}{\cAoMan}$. This means that the Weil--Berezin functor
of $G$ takes values in the category of $A_0$\nbd smooth Lie groups.
\end{corollary}

\begin{proof}
This is an immediate consequence of prop.\ \ref{prop:s-emb}.
\end{proof}

For more information on group objects and product preserving functors, see \cite{Vistoli}.

\subsection{Representability of the Weil--Berezin functor} \label{subsec:representability}

Next definition is the natural generalization of the classical one
to the Weil--Berezin functor setting.

\begin{definition}
We say that a functor
\[
    \cF \colon \SWA \to \cAoMan
\]
is representable if there exists a supermanifold $M_\cF$ such that
$\cF \isom (M_\cF)_{(\blank)}$ in $\dfunct{\SWA}{\cAoMan}$.
\end{definition}

Notice that we are abusing the category terminology, that
considers a functor $\cF$ to be representable if and only if $\cF$
is isomorphic to the $\Hom$ functor.

\medskip

Due to theorem \ref{theor:full_and_faithful}, if a functor $\cF$ is
representable, then the supermanifold $M_\cF$ is unique up to
isomorphism. Next example shows that there exists non representable
functors.

\begin{example}
Consider the constant functor $\SWA \to \cAoMan$ defined as
$A \mapsto \K$ on the objects ($\K \isom A/\nil{A}$ is an $A$\nbd
module) and $\rho \mapsto \id_\K$ on the morphisms. This functor is
not representable in the sense explained above.
\end{example}

In this subsection we look for conditions ensuring the
representability for a functor
$\cF\colon \SWA \to \cAoMan$.

Since $\cF(\K)$ is a manifold, we can consider an open set $U
\subseteq \cF(\K)$. If $A$ is a super Weil algebra and
$\funcpt{\pr_A} \coloneqq \cF(\pr_A)$, where $\pr_A$ is the
projection $A \to \K$, $\funcpt{\pr_A}^{-1}(U)$ is an open $A_0$\nbd
submanifold of $\cF(A)$. Moreover, if $\rho \colon A \to B$ is a
superalgebra map, since $\pr_B \circ \rho = \pr_A$, $\funcpt{\rho}
\coloneqq \cF(\rho)$ can be restricted to
\[
    \subfunc{\funcpt{\rho}}{\funcpt{\pr_A}^{-1}(U)} \colon \funcpt{\pr_A}^{-1}(U) \to \funcpt{\pr_B}^{-1}(U) \text{.}
\]
We can hence define the functor
\[
    \begin{aligned}
        \subfunc{\cF}{U} \colon \SWA &\to \cAoMan \\
        A &\mapsto \funcpt{\pr_A}^{-1}(U) \\
        \rho &\mapsto \subfunc{\funcpt{\rho}}{\funcpt{\pr_A}^{-1}(U)} \text{.}
    \end{aligned}
\]

\begin{proposition}[Representability]
A functor
\[
    \cF \colon \SWA \to \cAoMan
\]
is representable if and only if there exists an open cover
$\set{U_i}$ of $\cF(\K)$ such that $\subfunc{\cF}{U_i} \isom
(\bV_i)_{(\blank)}$ with $\bV_i$ superdomains in a fixed $\K^{p|q}$.
\end{proposition}

\begin{proof}
The necessity is clear due to the very definition of supermanifold.
Let us prove sufficiency. We have to build a supermanifold structure
on the topological space $\topo{\cF(\K)}$. Let us denote by
$(h_i)_{(\blank)} \colon \cF_{U_i} \to (\bV_i)_{(\blank)}$ the natural
isomorphisms in the hypothesis. On each $U_i$, we can put a
supermanifold structure $\bU_i$, defining the sheaf $\sheaf_{\bU_i}
\coloneqq [(h_i^{-1})_\K]_* \sheaf_{\bV_i}$. Let $k_i$ be the
isomorphism $\bU_i \to \bV_i$ and $(k_i)_{(\blank)}$ the
corresponding natural transformation. If $U_{i,j} \coloneqq U_i \cap
U_j$, consider the natural transformation $(h_{i,j})_{(\blank)}$
defined by the composition
\[
    (k_i^{-1})_{(\blank)} \circ (h_i)_{(\blank)} \circ (h_j^{-1})_{(\blank)} \circ (k_j)_{(\blank)}
    \colon (U_{i,j},\restr{\sheaf_{\smash{\bU_j}}}{U_{i,j}})_{(\blank)}
    \to (U_{i,j},\restr{\sheaf_{\smash{\bU_i}}}{U_{i,j}})_{(\blank)}
\]
where in order to avoid heavy notations we didn't explicitly
indicate the appropriate restrictions. Each $(h_{i,j})_{(\blank)}$
is a natural isomorphism in $\dfunct{\SWA}{\cAoMan}$ and, due to
lemma \ref{lemma:A0_smooth_nat_tr}, it gives rise to a supermanifold
isomorphism
\[
    h_{i,j} \colon (U_{i,j},\restr{\sheaf_{\smash{\bU_j}}}{U_{i,j}}) \to (U_{i,j},\restr{\sheaf_{\smash{\bU_i}}}{U_{i,j}}) \text{.}
\]
The $h_{i,j}$ satisfy the cocycle conditions $h_{i,i} = \id$ and
$h_{i,j} \circ h_{j,k} = h_{i,k}$ (restricted to $U_i \cap U_j \cap
U_k$). This follows from the analogous conditions satisfied by
$(h_{i,j})_A$ for each $A\in\SWA$. The supermanifolds $\bU_i$ can
hence be glued (for more information about the construction of a
supermanifold by gluing see for example \cite[ch.~2]{DM} or
\cite[\S~4.2]{Varadarajan}). Denote by $M_\cF$ the manifold thus
obtained.
 Moreover it is clear that $\cF$
is represented by the supermanifold $M_\cF$. Indeed, one can check
that the various $(h_i)_{(\blank)}$ glue together and give a natural
isomorphism $h_{(\blank)} \colon \cF \to (M_\cF)_{(\blank)}$.
\end{proof}

\begin{remark}
The supermanifold $M_\cF$ admits a more synthetic characterization.
In fact it is easily seen that $\topo{M_\cF} \coloneqq
\topo{\cF(\K)}$ and
\[
    \sheaf_{M_\cF}(U) \coloneqq \Hom_{\dfunct{\SWA}{\cAoMan}} \big( \subfunc{\cF}{U}, \K^{1|1}_{(\blank)} \big) \text{.}
\]
\end{remark}

\subsection{The functors of $\Lambda$\nbd points}
\label{lambdapoints}

In this subsection we want to give a brief exposition of the
original approach of A. S. Shvarts and A. A. Voronov (see
\cite{Shvarts,Voronov}). In their work they considered only
Grassmann algebras instead of all super Weil algebras. There are
some advantages in doing so: Grassmann algebras are many fewer,
moreover, as we noticed in remark \ref{remark:localalg}, they are
the sheaf of the super domains $\K^{0|q}$ and so the restriction to
Grassmann algebras of the local functors of points can be considered
as a true restriction of the functor of points. Finally the use of
Grassmann algebras is also used by A. S. Shvarts to formalize the
language commonly used in physics.

On the other hand the use of super Weil algebras has the advantage
that we can perform differential calculus on the
Weil--Berezin functor as we shall see in section \ref{sec:diff_calc}.
Indeed prop.\ \ref{prop:distributions} is valid only for the
Weil--Berezin functor
approach, since not every point supported distribution can be
obtained using only Grassmann algebras.
Also theorem \ref{theor:transitivity} and its consequences are valid
only in this approach, since purely even Weil algebras are
considered.

\medskip

If $M$ is a supermanifold and $\Gras$ denotes the category of
Grassmann algebras, we can consider the two functors
\begin{align*}
    &\begin{aligned}
        \Gras &\to \Sets \\
        \Lambda &\mapsto M_\Lambda
    \end{aligned} &
    \begin{aligned}
        \Gras &\to \cAoMan \\
        \Lambda &\mapsto M_\Lambda
    \end{aligned}
\end{align*}
in place of those
introduced by eq.\ \eqref{eq:A-point_functor} and eq.\
\eqref{eq:WB_functor} respectively.
As in the case of $A$\nbd points, with a slight abuse of notation
we denote by $M_\Lambda$ the $\Lambda$\nbd points
for each of the two different functors. What we have seen previously 
still remains
valid in this setting, provided we substitute systematically $\SWA$
with $\Gras$; in particular theorems \ref{theor:azerolinear} and
\ref{theor:full_and_faithful} still hold true. They are based on
prop.\ \ref{prop:formal_series} and lemma
\ref{lemma:A0_smooth_nat_tr} that we state here in their original
formulation as it is contained in \cite{Voronov}.

\begin{proposition}
The set of natural transformations between $\Lambda \mapsto
\K^{p|q}_{\Lambda}$ and $\Lambda \mapsto \K^{m|n}_{\Lambda}$ is in
bijective correspondence with
\[
    \big( \fseries{p}{q}(\K^p) \big)_0^m \times \big( \fseries{p}{q}(\K^p) \big)_1^n \text{.}
\]
A natural transformation comes from a supermanifold morphism $\K^{p|q} \to
\K^{m|n}$ if and only if it is $\Lambda_0$\nbd smooth for each
Grassmann algebra $\Lambda$.
\end{proposition}

\begin{proof}
The proof is the same as in prop.\ \ref{prop:formal_series} and lemma
\ref{lemma:A0_smooth_nat_tr}. The only difference is in the first
proof. Indeed the algebra \eqref{eq:hatA} is not a Grassmann
algebra. So, if $A = \extn{n} = \ext{\epsilon_1,\ldots,\epsilon_n}$,
we have to consider
\[
    \hat{A} \coloneqq \extn{2p(n-1)+q} = \ext{\eta_{i,a},\xi_{i,a},\zeta_j}
\]
($1 \leq i \leq p$, $1 \leq j \leq q$, $1 \leq a \leq n-1$). A
$\extn{n}$\nbd point can be written as
\[
    x_{\extn{n}} = \left( u_1 + \sum_{a<b} \epsilon_a \epsilon_b k_{1,a,b},\ldots,u_p + \sum_{a<b} \epsilon_a \epsilon_b k_{p,a,b},
    \kappa_1,\ldots,\kappa_q \right)
\]
with $u_i \in \K$, $k_{i,a,b} \in (\extn{n})_0$ and $\kappa_j \in
(\extn{n})_1$. Its image under a natural transformation can be
obtained taking the image of the $\extn{2p(n-1)+q}$\nbd point
\[
    y_{\red{x_{\extn{n}}}} \coloneqq \left( u_1 + \sum_{a=1}^{n-1} \eta_{1,a}\xi_{1,a},\ldots,u_p + \sum_{a=1}^{n-1} \eta_{p,a}\xi_{p,a},\zeta_1,\ldots,\zeta_q \right)
\]
and applying the map $\extn{2p(n-1)+q} \to \extn{n}$
\[
    \left\{
    \begin{aligned}
    \eta_{i,a} &\mapsto \epsilon_a \\
    \xi_{i,a} &\mapsto \textstyle\sum_{b>a} \epsilon_b k_{i,a,b} \\
    \zeta_j &\mapsto \kappa_j
    \end{aligned}
    \right.
\]
to each component. The nilpotent part of each even component of
$y_{\red{x_{\extn{n}}}}$ can be viewed as a formal scalar product
between $(\eta_{i,1},\ldots,\eta_{i,n-1})$ and
$(\xi_{i,1},\ldots,\xi_{i,n-1})$. This is stable under formal
rotations, where $\eta_{i,a}$ and $\epsilon_b$ have to be thought
as coordinates in an $n-1$ dimensional space and the rotation
applies to them.
The same must be for its image. 
So $\eta_{i,a}$ and
$\xi_{i,a}$ can occur in the image only as a polynomial in $\sum_a
\eta_{i,a} \xi_{i,a}$. In other words the image of
$y_{\red{x_{\extn{n}}}}$ (and then of $x_{\extn{n}}$) is polynomial
in the nilpotent part of the coordinates.
\end{proof}

\subsection{Batchelor's approach} \label{Batchelor}

In this section we want to examine Batchelor's approach to supergeometry
(see \cite{Batchelor}, Sec. 1) and relate it with what was done
in the previous sections.

\medskip

In such an approach, a supermanifold is modelled on what
is called a $(r,s)$-dimensional super Euclidean space $E^{r,s}$,
which is given a non Hausdorff topology. Despite
the apparently very different setting, the choice of
an appropriate set of morphisms, makes the category of
such objects equivalent to the category of supermanifolds
as we have defined in Sec. \ref{subsec:SMan}. 
We shall also see the interpretation of
Batchelor's definition in term of our local functors of points.

\medskip

Throughout this section, by a supermanifold we mean a
{\sl differentiable supermanifold}, since Batchelor's approach
cannot be extended (as it is) to the holomorphic category.

\begin{definition}
Let $r$, $s<L$ be integers. We define 
\textit{$(r,s)$-dimensional super Euclidean space}  $E^{r,s}$ to be:
$$
E^{r,s}=\Lambda_\R(\theta_1 \dots \theta_L)^r_0 \oplus 
\Lambda_\R(\theta_1 \dots \theta_L)^s_1
$$
where $\Lambda_\R(\theta_1 \dots \theta_L)$ as usual denotes
the real exterior algebra in the variables $\theta_1 \dots \theta_L$,
while $\Lambda_\R(\theta_1 \dots \theta_L)^r_0 $ means we are taking
the direct product of $r$ copies of the even part of
$\Lambda_\R(\theta_1 \dots \theta_L)$ and similarly for
$\Lambda_\R(\theta_1 \dots \theta_L)^s_1$.
\end{definition}

We endow $E^{r,s}$ with the following topology: $U \subset E^{r,s}$
is \textit{open} if and only if $U=\epsilon^{-1}(V)$, for $V$ open
in $\R^r$ where
$$
\begin{array}{cccc}
\epsilon: & E^{r,s} & \lra & \R^r \\
&(u_1 \dots u_r,v_1 \dots v_s) & \mapsto & (u_1^0, \dots u_r^0)
\end{array}
$$
where $u_i^0$ is the component of $\Z$-degree 0 of the element
$u_i \in \Lambda_\R(\theta_1 \dots \theta_L)_0$.

\medskip

It is clear that the topology so defined is non Hausdorff. In fact
 we have that, as vector spaces, $E^{r,s} \cong \R^{N}$, for
$N>r$, hence if we define the topology via the projection
$\epsilon:E^{r,s}  \lra  \R^r $, 
we are unable to separate points which
lie in the same fiber above the same point in $\R^r$. 
In other words, two points $(u_1 \dots u_r,v_1 \dots v_s)$,
$(u_1' \dots u_r',v_1' \dots v_s') \in E^{r,s} $ with
$(u_1^0 \dots u_r^0)=({u_1^0}' \dots {u_r^0}')$ have the
same neighbourhoods.

\medskip

Let us clarify in the next example the relation of 
$E^{r,s}$ with the $\Lambda$-points and more in general
with the $T$-points of a supermanifold.

\begin{example}
Let us consider the superspace $\R^{r|s}$ as in 
example \ref{superspace-rs} and the $T$-points
of $\R^{r|s}$ for $T=\R^{0|L}$. Notice that topologically 
$|T|=\R^0$, that is, $|T|$ is a point. 
By the Chart theorem 
\ref{prop:morphisms} we have that $\R^{r|s}(T)$ is in bijective
correspondence with the $(r|s)$-uples of elements in
$\cO(T)_0^r \times \cO(T)_1^s$. In other words:
$$
\begin{array}{rl}
\R^{r|s}(\R^{0|L})&:=\Hom(\R^{0|L},\R^{r|s})=
\Hom(\cO(\R^{r|s}),\cO(\R^{0|L}))= \\ \\
&=\left\{\phi:C^{\infty}(\R^r) \otimes \Lambda_\R(\eta_1 \dots \eta_s)
\lra \Lambda_\R(\theta_1 \dots \theta_L) \right\} \\ \\
&=\left\{(u_1 \dots u_r,\nu_1 \dots \nu_s) \in 
\Lambda_\R(\theta_1 \dots \theta_L)_0^r \times 
\Lambda_\R(\theta_1 \dots \theta_L)_1^s
\right\} \cong E^{r,s}.
\end{array}
$$
Hence we have that set-theoretically $E^{r,s}$ are  the $\R^{0|L}$-points
of $\R^{r|s}$.

Consequently in the language of $\Lambda$-points of the previous
sections, $E^{r,s}$ are the 
$\Lambda_\R(\theta_1 \dots \theta_L)$-points of the 
supermanifold $\R^{r|s}$.

\end{example}

We shall now give the definition of supermanifolds according
to Batchelor.

\begin{definition}
Let $|M|$ be a topological space. We define a 
\textit{super-Euclidean chart} on $|M|$ as a pair $(U,\phi)$
where $U \subset |M|$ and $\phi:U \lra E^{r,s}$ is an homeomorphism
onto its image.
We say that $\{(U_\al, \phi_\al)\}$ is a smooth superatlas 
on $|M|$ if the $U_\al$ form an open cover of $|M|$ and
$$
\phi_\al \cdot \phi_\be^{-1}:\phi_\be(U_\al \cap U_\be)
\lra \phi_\al(U_\al \cap U_\be)
$$ 
is a {\sl superdiffeomorphism}.
  
A \textit{$B$-supermanifold} is a topological space together
with a maximal smooth superatlas.

\end{definition} 

In this definition we have not specified what a superdiffeomorphism
is. Certainly it is a diffeomorphism in the ordinary sense
(we are in $\R^N$ for some $N$), but it is necessary to
require further conditions in order to obtain the correct set of arrows,
that will give the equivalence of categories between supermanifolds
and $B$-supermanifolds.
Given the scope of the present paper, we are unable to provide a
characterization of the morphisms of $B$-supermanifolds
and so to properly define the
category of $B$-supermanifolds.
Furthermore it must be noticed the role of $L$ in the construction.
For $L'<L$ we obtain a subcategory of $B$-supermanifolds, which
embeds into the category constructed above. Consequently in order
to take into account all the possible values of $L$ it is necessary
to consider a direct limit. We shall not pursue this point furtherly
in our note.

\medskip

Despite our necessarily short treatment,
the next observation will
give an overview of how this construction stands with
respect to the others we have examined so far.

\begin{observation}
Let $M=(|M|,\cO_M)$ be a supermanifold defined as a superspace
together with a local model. The functor of points of $M$ is
by definition a functor:
\[
\FOP{M} \colon \SMan\op \to \Sets, 
\qquad \fop{M}{S} \coloneqq \Hom(S,M).
\]
One should notice that the arriving category is $\Sets$,
which has a very simple structure. We know that $\FOP{M}$
characterizes $M$, in the sense that supermanifolds and their
functors of points correspond bijectively to each other
(Yoneda's Lemma).
The functor of $\Lambda$-points on the other hand,
weakens considerably the category we start from, 
in the sense that we are considering only 
supermanifolds with underlying topological space made by just
one point: namely those supermanifolds whose superalgebra
of global sections is a Grassmann algebra.
If we want such a functor to characterize the supermanifold, 
we are forced to increase the hypothesis on the arrival category,
namely we need to ask that the set $M_\Lambda$ is in $\cAoMan$, the category
of $\cA_0$-manifolds, and consequently the morphisms appear to
be more complicated as we have seen in the previous sections.
In Batchelor's approach this strategy reaches its limit: 
we are taking into exam 
just one $\Lambda$-point for a special Grassmann algebra, which
has room enough to accomodate the odd dimensions: this the meaning
of the condition on the integer $L$ to be greater than the odd dimension $s$.
Consequently in order to achieve the equivalence
of categories, in other words, in order for this special $\Lambda$-point to
completely characterize the supermanifold, it is necessary to
add extra hypotheses on the arrival category, and this is the
meaning of the complicated non Hausdorff local model and of the
difficulty in describing the correct set of arrows, which we have not detailed.
The Weil-Berezin functor of points represents, in our opinion,
a good compromise between the functor of points $\FOP{M}$, the
functor of $\Lambda$-points and Batchelor's approach. 
The morphisms appear in a natural and reasonably simple
form, while we still have not to deal with the category of $\SMan$.

\end{observation}

In the next section we want to show why we believe the Weil-Berezin functor 
of points shows a definite advantage, with respect to the other
equivalent descriptions we have
examined so far, when dealing with differential geometry issues.

\section{Applications to differential calculus} \label{sec:diff_calc}

In this section we discuss some aspects of super differential
calculus on supermanifolds using the language of the Weil--Berezin
functor. In particular we establish a relation between the
$A$\nbd points of a supermanifold $M$ and the finite support
distributions over it, which
play a crucial role in Kostant's
seminal approach to supergeometry.

We also prove the super version of the Weil transitivity
theorem, which is a key tool for the study of the infinitesimal
aspects of supermanifolds, and we apply it in order to define
the ``tangent functor'' of $A\mapsto M_A$.

\subsection{Point supported distributions and $A$\nbd points} \label{subsec:distributions}

In this subsection we want to introduce and discuss
Kostant's approach (see \cite{Kostant}) using
the Weil--Berezin functor formalism.

\medskip

Let $(\topo{M},\sheaf_M)$ a supermanifold of dimension $p|q$ and $x
\in \topo{M}$. As in \cite[\S~2.11]{Kostant}, let us
consider the distributions with support at $x$.

\begin{definition}
If $\stalk{M}{x}'$ is the algebraic dual of the stalk at $x$, the
\emph{distributions with support at $x$ of order $k$} are defined as:
\[
    \distr[k]{M}{x} \coloneqq \Set{v \in \stalk{M}{x}' |
v \big( \maxid_x^{k+1} \big)=0}
\]
where $\maxid_x$ is, as usual, the maximal ideal of the germs of
sections that are zero if evaluated at $x$. Clearly $\distr[k]{M}{x}
\subseteq \distr[k+1]{M}{x}$. The
\emph{distributions with support at $x$} are given by the union
\[
    \distr{M}{x} \coloneqq \bigcup_{k = 0}^\infty \distr[k]{M}{x} \text{.}
\]
\end{definition}

\begin{observation}
The distributions $\distr[k]{M}{x}$
form a super vector space: an even distribution is
$0$ on an odd germ and vice versa. If
$x_1,\ldots,x_p,\theta_1,\ldots,\theta_q$ are coordinates in a
neighbourhood of $x$, a distribution of order $k$ is of the form
\[
    v = \sum_{\substack{\nu \in \N^p \\
    J \subseteq \set{1,\ldots,q} \\
    \abs{\nu} + \abs{J} \leq k}} a_{\nu,J} \;
    \ev_x \frac{\partial^{\abs{\nu}}}{\partial x^\nu}
    \frac{\partial^{\abs{J}}}{\partial \theta^J}
\]
with $a_{\nu,J} \in \K$. {This is immediate since we have the
following isomorphisms:
\[
    \distr[k]{M}{x}
    \isom \big( \sheaf_{\topo{M},x} \otimes \ext{\theta_1,\dots,\theta_q} \big)^*
    \isom \sheaf_{\topo{M},x}^{*} \otimes {\Lambda(\theta_1,\dots,\theta_q)}^*
\]
and $ \sheaf_{\topo{M},x}^{*} =\sum a_{\nu,J} \;
\ev_x \frac{\partial^{\abs{\nu}}}{\partial x^\nu}$ because of
the classical theory.}
\end{observation}

\begin{proposition} \label{prop:distributions}
Let $A$ be a super Weil algebra and $A^*$ its dual. Let
\[
    x_A \colon \stalk{M}{x} \to A
\]
be an $A$\nbd point near $x \in \topo{M}$. If $\omega \in A^*$, then
\[
    \omega \circ x_A \in \distr{M}{x} \text{.}
\]
Moreover each element of $\distr[k]{M}{x}$ can be obtained in this
way with
\[
    A = \stalk{M}{x} / \maxid_x^{k+1} \isom \kspoly{p}{q} / \maxid_0^{k+1}
\]
(see lemma \ref{lemma:SWA}).
\end{proposition}

\begin{proof}
If $A$ has height $k$, since $x_A(\maxid_x) \subseteq \nil{A}$, $\omega \circ x_A \in
\distr[k]{M}{x}$. If vice versa $v \in \distr[k]{M}{x}$, it factorizes through
\[
    \stalk{M}{x} \stackrel{\pr}{\to} \stalk{M}{x} / \maxid_x^{k+1} \stackrel{\omega}{\to} \K
\]
for a suitable $\omega$.
\end{proof}

In the next observation we relate the finite support distributions,
together with their interpretation via the Weil--Berezin functor,
with the tangent superspace.

\begin{observation} \label{obs:tangent_bundle}
Let us first recall that the tangent superspace to a supermanifold
$M$ at a point $x$ is the super vector space consisting of all the
derivations of the stalk at $x$:
\[
    T_x(M) \coloneqq \set{ v \colon \sheaf_{M,x} \to \K | \text{$v$ is a derivation} } \text{.}
\]

As in the classical setting we can recover the tangent space by
using the algebra of \emph{super dual numbers}. Let us consider $A =
\K(e,\epsilon)=\K[e,\epsilon]/ \langle e^2, e\epsilon,
\epsilon^2 \rangle$ be the super Weil algebra of super dual
numbers (see example \ref{example:SDN}). If $x_A \in M_{A,x}$ and
$s,t \in \stalk{M}{x}$, we have
\[
    x_A(st) = \ev_x(st) + x_e(st) e + x_\epsilon(st) \epsilon
\]
with $x_e,x_\epsilon \colon \stalk{M}{x} \to \K$. On the other hand
\begin{align*}
    x_A(st)
    &= x_A(s) x_A(t) \\
    &= \ev_x(s) \ev_x(t)
        + \big( x_e(s) \ev_x(t) + \ev_x(s) x_e(t) \big) e \\
        &\qquad + \big( x_\epsilon(s) \ev_x(t) + \ev_x(s) x_\epsilon(s) \big) \epsilon \text{.}
\end{align*}
Then $x_e$ (resp.\ $x_\epsilon$) is a derivation of the stalk that
is zero on odd (resp.\ even) elements and so $x_e \in T_x(M)_0$
(resp.\ $x_\epsilon \in T_x(M)_1$). The map
\[
    \begin{aligned}
        T(M) \coloneqq \bigsqcup_{x \in \topo{M}} T_x(M) &\to M_{\K(e,\epsilon)} \\
        v_0 + v_1 &\mapsto \ev_x + v_0 e + v_1 \epsilon
    \end{aligned}
\]
(with $v_i \in T_x(M)_i$) is an isomorphism of vector bundles over
$\red{M} \isom M_\K$, where $\red{M}$ is the
classical manifold associated with $M$, as in subsection
\ref{subsec:SMan} (see also \cite[ch.~8]{KMS} for an exhaustive
exposition in the classical case). The reader should not confuse
$T(M)$, which is the classical bundle obtained by the union of all
the tangent superspaces at the different points of $\topo{M}$, with
$\sheaf[T]_M$ which is the super vector bundle of all the derivations of
$\sheaf_M$.
\end{observation}

\subsection{Transitivity theorem and applications} \label{subsec:trans_th}

We now want to give a brief account on how we can
perform differential calculus using the language of $A$\nbd
points. The essential ingredient is the super version of the
transitivity theorem that we discuss below.

\medskip

In the following, when classical smooth (resp.\
holomorphic) manifolds are considered, $\sheaf$ denotes the
corresponding sheaf of smooth (resp.\ holomorphic) functions.

\begin{theorem}[Weyl transitivity theorem] \label{theor:transitivity}
Let $M$ be a smooth (resp.\ holomorphic) supermanifold, $A$ a super
Weil algebra and $B_0$ a purely even Weil algebra, both real (resp.\
complex). Then
\[
    {(M_A)}_{B_0} \isom M_{A \otimes B_0}
\]
as $(A_0 \otimes B_0)$\nbd manifolds.
\end{theorem}

\begin{proof}
Let $\sheaf_{M_A}$ and $\sheaf_{M_A}^A$ be the sheaves of smooth
(resp.\ holomorphic) maps from the classical manifold $M_A$ to $\K$
and $A$ respectively. Clearly $\sheaf_{M_A}^A \isom A \otimes
\sheaf_{M_A}$ through the map $f \mapsto \sum_i a_i \otimes
\pair{a_i^*}{f}$, where $\set{a_i}$ is a homogeneous basis of $A$.

If $x_A \in M_A$, let
\[
    \begin{aligned}
        \tau_{x_A} \colon \stalk{M}{\red{x_A}} &\to \stalk{M_A}{x_A}^A \isom A \otimes \stalk{M_A}{x_A} \\
        \germ{s}_{\red{x_A}} &\mapsto \germ{\hat{s}}_{x_A}
    \end{aligned}
\]
where, if $s \in \sheaf_M(U)$ and $\red{x_A} \in U$,
\[
    \hat{s} \colon y_A \mapsto y_A(s)
\]
for all $y_A \in M_A$ such that $\red{y_A} \in U$ (it is not
difficult to show that this map descends to a map between stalks).

Recalling that
\begin{align*}
    {(M_A)}_{B_0} &\coloneqq \bigsqcup_{x_A \in M_A} \Hom_{\SAlg}(\stalk{M_A}{x_A},B_0) \\
    M_{A \otimes B_0} &\coloneqq \bigsqcup_{x \in \topo{M}} \Hom_{\SAlg}(\stalk{M}{x},A \otimes B_0) \text{,}
\end{align*}
we can define a map
\[
    \begin{aligned}
        \xi \colon {(M_A)}_{B_0} &\to M_{A \otimes B_0} \\
        X &\mapsto \xi(X)
    \end{aligned}
\]
setting
\[
    \xi(X) \colon \germ{s}_{\rred{X}} \mapsto (\id_A \otimes X)\tau_{\red{X}}(\germ{s}_{\rred{X}}) \text{.}
\]
This definition is well-posed since $\xi(X)$ is a superalgebra map,
as one can easily check.

Fix now a chart $(U,h)$, $h \colon U \to \K^{p|q}$, in $M$ and
denote by $(U_A,h_A)$, $\big({(U_A)}_{B_0}, {(h_A)}_{B_0}\big)$
and $(U_{A\otimes B_0}, h_{A\otimes B_0})$ the corresponding charts
lifted to $M_A$, ${(M_A)}_{B_0}$ and $M_{A \otimes B_0}$
respectively. If $\set{e_1,\ldots,e_{p+q}}$ is a homogeneous basis
of $\K^{p|q}$, we have (here, according to observation
\ref{obs:coordinates}, we tacitly use the identification
$\K^{p|q}_A \isom (A \otimes \K^{p|q})_0$)
\[
    \begin{aligned}
        {(h_A)}_{B_0} \colon {(U_A)}_{B_0} &\to (A \otimes B_0 \otimes \K^{p|q})_0 \\
        X &\mapsto \sum_{i,j} a_i \otimes X \big( h_A^*(a_i^* \otimes e_j^*) \big) \otimes e_j
    \end{aligned}
\]
and
\[
    \begin{aligned}
        h_{A\otimes B_0} \colon U_{A \otimes B_0} &\to (A \otimes B_0 \otimes \K^{p|q})_0 \\
        Y &\mapsto \sum_k Y \big( h^*(e_k^*) \big) \otimes e_k \text{.}
    \end{aligned}
\]
Then, since
\begin{equation*}
    \xi(X) \big( h^*(e_k^*) \big)
    = (\id \otimes X) \big( \widehat{\smash{h^*}(e_k^*)} \big)
    = (\id \otimes X) \big( \textstyle\sum_i a_i \otimes h_A^*(a_i^* \otimes e_k^*) \big) \text{,}
\end{equation*}
we have
\[
    h_{A\otimes B_0} \circ \xi \circ {(h_A)}_{B_0}^{-1} = \id_{{(h_A)}_{B_0}({(U_A)}_{B_0})} \text{.}
\]
This entails in particular that $\xi$ is a local $(A_0 \otimes
B_0)$\nbd diffeomorphism. The fact that it is a global
diffeomorphism follows noticing that it is fibered over the
identity, being
\[
    \xi(X) \in M_{A \otimes B_0,\smash{\rred{X}}} \text{.}
    \qedhere
\]
\end{proof}

From now on we assume all supermanifolds to be smooth.

\medskip

We want to briefly explain some applications of the
Weil transitivity theorem to the smooth category. Let $M$
be a smooth supermanifold and let $A$ be a real super Weil algebra.
As we have seen in the previous section, $M_A$ has a natural
structure of classical smooth manifold and, due to prop.\
\ref{prop:smoothcase}, we can identify $M_A$ with the space of
superalgebra maps $\sheaf(M) \to A$.

\begin{definition}
If $x_A \in M_A$, we define the space of \emph{$x_A$\nbd linear
derivations} of $M$ (\emph{$x_A$\nbd derivations} for short) as the
$A$\nbd module\footnote{%
    We recall that if $V = V_0 \oplus V_1$ and
    $W = W_0 \oplus W_1$ are super vector spaces then $\HOM(V,W)$
    denotes the super vector space of all linear
    morphisms between $V$ and $W$ with the gradation $\HOM(V,W)_0
    \coloneqq \Hom(V_0,W_0) \oplus \Hom(V_1,W_1)$, $\HOM(V,W)_1
    \coloneqq \Hom(V_0,W_1) \oplus \Hom(V_1,W_0)$.
}
\begin{align*}
    \Der_{x_A} \big( \sheaf(M),A \big)
    \coloneqq \Big\{\, & X \in \HOM \big( \sheaf(M),A \big) \,\Big|\, \forall s,t \in \sheaf(M) , \\
    & X(st)= X(s) x_A(t) + (-1)^{\p{X}\p{s}}x_A(s)X(t) \,\Big\} \text{.}
\end{align*}
\end{definition}

\begin{proposition}
The tangent superspace at $x_A$ in $M_A$ canonically identifies with
$\Der_{x_A} \big( \sheaf(M), A \big)_0$.
\end{proposition}

\begin{proof}
If $\R(e)$ is the algebra of dual number (see example
\ref{example:DN}), $(M_A)_{\R(e)}$ is isomorphic, as a vector
bundle, to the tangent bundle $T(M_A)$, as we have seen in observation
\ref{obs:tangent_bundle}. Due to theorem
\ref{theor:transitivity}, we thus have an isomorphism
\[
    \xi \colon T(M_A) \isom {(M_A)}_{\R(e)} \to M_{A\otimes\R(e)} \text{.}
\]
On the other hand, it is easy to see that $x_{A\otimes \R(e)} \in
M_{A\otimes\R(e)}$ can be written as $x_{A\otimes \R(e)} = x_A
\otimes 1 + v_{x_A} \otimes e$, where $x_A \in M_A$ and $v_{x_A}
\colon \sheaf(M) \to A$ is a parity preserving map satisfying the
following rule for all $s,t \in \sheaf(M)$:
\[
    v_{x_A}(st) = v_{x_A}(s) \, x_A(t) + x_A(s) \, v_{x_A}(t) \text{.}
\]
Then each tangent vector on $M_A$ at $x_A$ canonically identifies a
even $x_A$\nbd derivation and, vice versa, each such derivation
canonically identifies a tangent vector at $x_A$.
\end{proof}

We conclude studying more closely the structure of $\Der_{x_A} \big(
\sheaf(M), A \big)$. The following proposition describes it
explicitly.

Let $K$ be a right $A$\nbd module and let $L$ be a left $B$\nbd
module for some algebras $A$ and $B$. Suppose moreover that an
algebra morphism $\rho \colon B \to A$ is given. One defines the
$\rho$\nbd tensor product $K \otimes_\rho L$ as the quotient of the
vector space $K \otimes L$ with respect to the equivalence relation
\[
    k \otimes b \cdot l \sim k \cdot \rho(b) \otimes l
\]
for all $k \in K$, $l \in L$ and $b \in B$.

Moreover, if $M$ is a supermanifold, we denote by $\sheaf[T]_M$ the
\emph{super tangent bundle} of $M$, i.~e.\ the sheaf defined by
$\sheaf[T]_M \coloneqq \Der(\sheaf_M)$.

\begin{proposition}
Let $M$ be a smooth supermanifold and let $x \in \topo{M}$. Denote
$\stalk[T]{M}{x}$ the germs of vector fields at $x$. One has the
identification of left $A$\nbd modules
\[
    \Der_{x_A} \big( \sheaf(M), A \big)
    \isom A \otimes T_{\red{x_A}}(M)
    \isom A \otimes_{x_A} \stalk[T]{M}{\red{x_A}} \text{.}
\]
\end{proposition}

This result is clearly local so that it is enough to prove it in the
case $M$ is a superdomain. Next lemma do this for the first
identification. The second descends from eq.\
\eqref{eq:x_A-derivation}, since $\stalk[T]{M}{\red{x_A}} =
\stalk{M}{\red{x_A}} \otimes T_{\red{x_A}}(M)$.

\begin{lemma}
Let $U$ be a superdomain in $\R^{p|q}$ with coordinate system
$\set{x_i,\theta_j}$, $A$ a super Weil algebra and $x_A \in U_A$. To
any list of elements
\[
    \f = (f_1,\ldots,f_p,F_1,\ldots,F_q)
    \qquad f_i, F_j \in A
\]
there corresponds a $x_A$\nbd derivation
\[
    X_{\f} \colon \sheaf(U) \to A
\]
given by
\begin{equation} \label{eq:x_A-derivation}
    X_{\f}(s) = \sum_i f_i \, x_A \left( \pd{s}{x_i} \right) + \sum_j F_j \, x_A \left( \pd{s}{\theta_j} \right) \text{.}
\end{equation}
$X_{\f}$ is even (resp.\ odd) if and only if the $f_i$ are even
(resp.\ odd) and the $F_j$ are odd (resp.\ even). Moreover any
$x_A$\nbd derivation is of this form for a uniquely determined $\f$.
\end{lemma}

\begin{proof}
That $X_{\f}$ is a $x_A$\nbd derivation is clear. That the family
$\f$ is uniquely determined is also immediate from the fact that
they are the value of $X_{\f}$ on the coordinate functions.

Let now $X$ be a generic $x_A$\nbd derivation. Define
\begin{align*}
    f_i &= X(x_i) \text{,} &
    F_j &= X(\theta_j) \text{,}
\end{align*}
and
\[
    X_{\f} = f_i \, x_A \circ \pd{}{x_i} + F_j \, x_A \circ \pd{}{\theta_j} \text{.}
\]
Let $D = X-X_{\f}$. Clearly $D(x_i) = D(\theta_j) = 0$. We now show
that this implies $D=0$. Let $s\in\sheaf(U)$. Due to lemma
\ref{lemma:polynomials}, for each $x \in U$ and for each integer
$k\in\N$ there exists a polynomial $P$ in the coordinates such that
$\germ{s}_x - \germ{P}_x \in \maxid_x^{k+1}$. Due to Leibniz rule
$D(s - P) \in \nil{A}^k$ and, since clearly $D(P)= 0$, $D(s)$ is in
$\nil{A}^k$ for arbitrary $k$. So we are done.
\end{proof}

The previous result gives the following corollary.

\begin{corollary}
We have the identification
\[
    T_{x_A} M_A
    \isom \big( A \otimes T_{\red{x_A}}(M) \big)_0
    \isom \big( A \otimes_{x_A} \stalk[T]{M}{\red{x_A}} \big)_0 \text{.}
\]
\end{corollary}


\begin{thebibliography}{99}

\bibitem{bcf}
L. Balduzzi, C. Carmeli, R. Fioresi.
{\em The local functors of points of supermanifolds}. Exp. Math. 
Vol. 28(3):201--217, 2010.

\bibitem{Batchelor}
Marjorie Batchelor.
\newblock Two approaches to supermanifolds.
\newblock {\em Trans. Amer. Math. Soc.}, 258(1):257--270, 1980.

\bibitem{BBHR}
C. Bartocci, U. Bruzzo, and D. Hern{\'a}ndez~Ruip{\'e}rez.
\newblock {\em The geometry of supermanifolds}, volume~71 of {\em Mathematics
  and its Applications}.
\newblock Kluwer Academic Publishers Group, Dordrecht, 1991.

\bibitem{Berezin}
F.~A. Berezin.
\newblock {\em Introduction to superanalysis}, volume~9 of {\em Mathematical
  Physics and Applied Mathematics}.
\newblock D. Reidel Publishing Co., Dordrecht, 1987.
\newblock Edited and with a foreword by A. A. Kirillov, With an appendix by V.
  I. Ogievetsky, Translated from the Russian by J. Niederle and R. Koteck{\'y},
  Translation edited by Dimitri Le{\u\i}tes.

\bibitem{BL}
F.~A. Berezin and D.~A. Le{\u\i}tes.
\newblock Supermanifolds.
\newblock {\em Dokl. Akad. Nauk SSSR}, 224(3):505--508, 1975.

\bibitem{CF}
C. Carmeli,
L. Caston and  R. Fioresi {\it  Mathematical Foundation of
Supersymmetry}, with an appendix with I. Dimitrov,
EMS Ser. Lect. Math., European
Math. Soc., Zurich 2011.

\bibitem{DeWitt}
B. DeWitt.
\newblock {\em Supermanifolds}.
\newblock Cambridge Monographs on Mathematical Physics. Cambridge University
  Press, Cambridge, 1984.

\bibitem{DG}
M. Demazure and P. Gabriel.
\newblock {\em Groupes alg\'ebriques. {T}ome {I}: {G}\'eom\'etrie alg\'ebrique,
  g\'en\'eralit\'es, groupes commutatifs}.
\newblock Masson \& Cie, \'Editeur, Paris, 1970.
\newblock Avec un appendice {{\i}t Corps de classes local} par Michiel
  Hazewinkel.

\bibitem{DM}
P. Deligne and J.~W. Morgan.
\newblock Notes on supersymmetry (following {J}oseph {B}ernstein).
\newblock In {\em Quantum fields and strings: a course for mathematicians, Vol.
  1, 2 (Princeton, NJ, 1996/1997)}, pages 41--97. Amer. Math. Soc., Providence,
  RI, 1999.

\bibitem{EH}
D. Eisenbud and J. Harris.
\newblock {\em The geometry of schemes}, volume 197 of {\em Graduate Texts in
  Mathematics}.
\newblock Springer-Verlag, New York, 2000.

\bibitem{FLV}
R.~Fioresi, M.~A. Lled{\'o}, and V.~S. Varadarajan.
\newblock The {M}inkowski and conformal superspaces.
\newblock {\em J. Math. Phys.}, 48(11):113505, 27, 2007.

\bibitem{Horm}
L.H\"ormander.
\newblock {\em An introduction to complex analysis in several variables. Third edition}.
\newblock North-Holland Mathematical Library, 7. North-Holland Publishing Co., Amsterdam, 1990

\bibitem{Kaup}
L.Kaup,  B.Kaup.
\newblock {\em Holomorphic functions of several variables. An introduction to the fundamental theory.}.
\newblock  de Gruyter Studies in Mathematics, 3. Walter de Gruyter  Co., Berlin, 1983.


\bibitem{KMS}
I. Kol{\'a}{\v{r}}, P.~W. Michor, and J. Slov{\'a}k.
\newblock {\em Natural operations in differential geometry}.
\newblock Springer-Verlag, Berlin, 1993.

\bibitem{Kostant}
B. Kostant.
\newblock Graded manifolds, graded {L}ie theory, and prequantization.
\newblock In {\em Differential geometrical methods in mathematical physics
  (Proc. Sympos., Univ. Bonn, Bonn, 1975)}, pages 177--306. Lecture Notes in
  Math., Vol. 570. Springer, Berlin, 1977.

\bibitem{Koszul}
J.-L. Koszul.
\newblock Differential forms and near points on graded manifolds.
\newblock In {\em Symplectic geometry ({T}oulouse, 1981)}, volume~80 of {\em
  Res. Notes in Math.}, pages 55--65. Pitman, Boston, MA, 1983.

\bibitem{Leites}
D.~A. Le{\u\i}tes.
\newblock Introduction to the theory of supermanifolds.
\newblock {\em Uspekhi Mat. Nauk}, 35(1(211)):3--57, 255, 1980.

\bibitem{MacLane}
S. MacLane.
\newblock {\em Categories for the working mathematician}.
\newblock Springer-Verlag, New York, 1971.
\newblock Graduate Texts in Mathematics, Vol. 5.

\bibitem{Manin}
Yu.~I. Manin.
\newblock {\em Gauge field theory and complex geometry}, volume 289 of {\em
  Grundlehren der Mathematischen Wissenschaften [Fundamental Principles of
  Mathematical Sciences]}.
\newblock Springer-Verlag, Berlin, 1988.
\newblock Translated from the Russian by N. Koblitz and J. R. King.

\bibitem{Rogers80}
A. Rogers.
\newblock A global theory of supermanifolds.
\newblock {\em J. Math. Phys.}, 21(6):1352--1365, 1980.

\bibitem{Rogers07}
A. Rogers.
\newblock {\em Supermanifolds}.
\newblock World Scientific Publishing Co. Pte. Ltd., Hackensack, NJ, 2007.
\newblock Theory and applications.

\bibitem{Shurygin}
V.~V. Shurygin.
\newblock The structure of smooth mappings over {W}eil algebras and the
  category of manifolds over algebras.
\newblock {\em Lobachevskii J. Math.}, 5:29--55 (electronic), 1999.

\bibitem{Shvarts}
A.~S. Shvarts.
\newblock On the definition of superspace.
\newblock {\em Teoret. Mat. Fiz.}, 60(1):37--42, 1984.

\bibitem{Varadarajan}
V.~S. Varadarajan.
\newblock {\em Supersymmetry for mathematicians: an introduction}, volume~11 of
  {\em Courant Lecture Notes in Mathematics}.
\newblock New York University Courant Institute of Mathematical Sciences, New
  York, 2004.

\bibitem{Vistoli}
A. Vistoli.
\newblock Notes on {G}rothendieck topologies, fibered categories and descent
  theory.
\newblock Preprint, {\ttfamily arXiv:math/0412512}, 2007.

\bibitem{Voronov}
A. Voronov.
\newblock Maps of supermanifolds.
\newblock {\em Teoret. Mat. Fiz.}, 60(1):43--48, 1984.

\bibitem{Weil}
A. Weil.
\newblock Th\'eorie des points proches sur les vari\'et\'es diff\'erentiables.
\newblock In {\em G\'eom\'etrie diff\'erentielle. {C}olloques {I}nternationaux
  du {C}entre {N}ational de la {R}echerche {S}cientifique, {S}trasbourg, 1953},
  pages 111--117. Centre National de la Recherche Scientifique, Paris, 1953.




\end{thebibliography}

\end{document}